\documentclass[12pt, twoside, english, headsepline]{article}

\usepackage[margin=30mm]{geometry}
\usepackage{tikz,graphicx}
\usetikzlibrary{arrows}
\usetikzlibrary{patterns}

\tikzset{
  mid arrow/.style={postaction={decorate,decoration={
        markings,
        mark=at position .5 with {\arrow[#1]{stealth}}
      }}},
}
\usetikzlibrary{decorations.markings} % per colorare le regioni di un grafico
\tikzset
 {every pin/.style = {pin edge = {<-}}, 
  > = stealth, 
  flow/.style = 
   {decoration = {markings, mark=at position #1 with {\arrow{>}}},
    postaction = {decorate}
   },
  flow/.default = 0.5,   
  main/.style = {line width=1pt}
 }

\usepackage{amsmath,amsthm,amssymb,amsfonts, mathtools, amsbsy, dsfont}
\usepackage[utf8]{inputenx}
\usepackage{xcolor}
\usepackage{microtype}
\usepackage{hyperref}
\hypersetup{linkcolor=blue}

\newcommand*\dif{\mathop{}\!\mathrm{d}}  % simbolo differenziale
\newcommand{\conv}{\mathop{\scalebox{2.8}{\raisebox{-0.2ex}{$\ast$}}}} % convoluzione grande (per fare convoluzioni multipler)
\newcommand{\sconv}{\mathop{\scalebox{2.0}{\raisebox{-0.2ex}{$\ast$}}}} % convoluzione grande (per fare convoluzioni multipler)
\newcommand\sbullet[1][.5]{\mathbin{\vcenter{\hbox{\scalebox{#1}{$\bullet$}}}}} % punto 
\allowdisplaybreaks % fa spezzare le formule ``equation'' tra le pagine

\newtheorem{theorem}{Theorem}[section] %[theorem]
\newtheorem{proposition}{Proposition}[section] %[theorem]
\newtheorem{lemma}[theorem]{Lemma}
\theoremstyle{definition} % plain
\newtheorem{definition}{Definition}[section] 
\newtheorem{example}{Example}[section]
\newtheorem{remark}{Remark}[section] %[theorem]
\numberwithin{equation}{section} % equazioni con numero interno alla sezione

\providecommand{\keywords}[1]
{
  \small	
  \textbf{\textit{Keywords: }} #1
}
\providecommand{\MSC}[1]
{
  \small	
  \textit{2020 MSC: } #1   
}

\title{Integrals of some compound processes}%Compound processes and their integrals
\author{Fabrizio Cinque$^1$ and Enzo Orsingher$^2$\\
        \small Department of Statistical Sciences, Sapienza University of Rome, Italy \\
        \small $^1$cinque.fabrizio@gmail.com $^2$enzo.orsingher@uniroma1.it
}

\begin{document}

\maketitle

\begin{abstract}
By means of the theory of compound (non-homogeneous) Poisson processes we define a general framework which includes most of the generalizations of the Poisson processes, both in the Skellam sense and in the space-fractional sense. We prove that Bernstein subordination is the only time change that leads to a compound Poisson process. We also consider the case of subordination with inverse Bernstein subordinators. Then, we focus on the integrals of compound processes. We give an explicit representation of the fractional integral of a compound Poisson process, showing that for fixed $t$ it is distributed as a compound Poisson random variable. We then extend this result to more general integral forms. Furthermore, we obtain some limit results, explicit forms of the iterated integrals and their governing equation. Finally, we study the integral of compound renewal processes in the Fourier-Laplace domain.
\end{abstract} \hspace{10pt}

\keywords{Non-homogeneous Poisson processes; Compound renewal processes; Integral of stochastic processes; Bernstein functions; Convolution-type derivatives}

\MSC{Primary 60G51, 60G55; Secondary 60G22}

%33C10 Bessel and Airy functions, cylinder functions
% 34A05 Explicit solutions, first integrals of ordinary differential equations
%35K25 Higher-order parabolic equations
% 60G55 Point processes
% 60G22 fractional processes, including Brownian motion
% 60G51 processes with independent increments< Levy processes

% ------  CORPO  ---------------------------------------------------------------------------------------

\section{Introduction}

The Poisson process has been generalized in many directions in recent decades. A first step was taken by replacing the time-derivative with the time-fractional derivative in the governing equation \cite{L2003}, which first introduced a fractional version of a point process. Another (equivalent) generalization was the replacement of the intertimes between successive events with Mittag-Leffler distributions \cite{MGS2004, MNV2011}. A completely different variation was obtained by acting on the shift space operator in the governing equation \cite{OP2012}, leading to the so-called space-fractional Poisson process and to its generalizations via Bernstein subordinators \cite{OT2015}. Another approach consisted in introducing the generalized counting process where multiple arrivals are permitted \cite{P1984} and by then applying a time-fractional operator to the governing equation \cite{DcMM2016}. Some authors also studied convolutional-type generalizations \cite{KM2021}. Alongside these extensions of the Poisson process, generalizations of the Skellam process (introduced in \cite{S1946}), obtained using the difference of various types of the above Poisson-type processes \cite{BS2024, CO202, KK2024, KV2019, MV2019}, appeared in the literature.

The first part of this work aims at the description of these processes in a more general framework based on the theory of compound Poisson processes, thus assuming jumps of arbitrary real-valued amplitude. Some particular cases of this approach already appeared in the literature, see for instance \cite{BM2012, BM2014, DcMZ2015, KK2023, S2013}. In particular, we study processes (say $M$) with independent increments and having the following infinitesimal behaviour, for $t\ge0, x\in\mathbb{R}$,
\begin{align*}\label{compoundPoissonInfinitesimaleIntroduzione}
	P\{M(t,t+\dif t] \in \dif x\} = \begin{cases}
		\begin{array}{l l}
			\lambda(t)\dif t F(\dif x) + o(\dif t), & x\in\mathbb{R}\setminus\{0\},\\
			1-\lambda(t)\dif t + \lambda(t)\dif t F(\dif 0) + o(\dif t), & x = 0,
		\end{array}
	\end{cases} 
\end{align*}
which turn out to be compound (non-homogeneous) Poisson processes. 
%These are among the most popular processes having good theoretical properties and a wide range of possible applications in several different fields.

In the second part of the work we focus on the integral of compound processes (say $M$) of the form $\int_0^t M(s) f(s) \dif s$, $t\ge0$, where $f$ is a suitable real function; in particular we deal with the Poisson and the renewal cases (we refer to \cite{B2022} for compound renewal processes and their asymptotics). For instance, with $t\ge0$ and $f(s) = (t-s)^\alpha,\ \alpha\ge0$, we are able to represent the fractional integral of the compound homogeneous Poisson process $M(t) = \sum_{k=1}^{N(t)}X_k, t\ge0,$ as a compound Poisson random variable,
 $$\int_0^t M(s)(t-s)^\alpha \dif s \stackrel{d}{=} \frac{t^{\alpha+1}}{\alpha+1} \sum_{k=1}^{N(t)} X_k U_k^{\alpha+1},\ \ \ t\ge0,\ \alpha\ge0,$$
 where $U_k\sim Unif(0,1)$ are i.i.d. r.v.s independent of the other random elements.
 
  In the literature, the integration of point processes has been investigated over the years because of its wide range of applications. Indeed, these integrals are meaningful for inventory and queuing systems and, in general, they can represent the total "cost" of the "elements" released in a time interval by the "subjects" appearing during that interval. For instance, in an environmental traffic analysis it may represent the total amount of emissions by the cars arriving and stopping at a traffic light; in a biological study, they can represent the food consumption of bacteria or animals that arrive in an area and reproduce randomly over a given time interval. Early works on these processes (and applications) are \cite{GMn1971, Mn1970, N1966, P1966, P1968}, about birth-death process, and \cite{W1971}, about compound Poisson processes. We refer to \cite{N1966} for a discussion on the applications and the study of the moments of integrated Poisson processes. Recently, several authors have devoted attention to this type of integral processes, focusing on some path-wise properties \cite{PC2026, PS2002}, on the renewal processes \cite {SVw2007}, on the fractional integrals \cite{CO202, OP2013}, on the Skellam process \cite{X2018}, and on the introduction of random times \cite{VK2024}. 
\\

The paper is organized as follows. In the second section we recall some properties of the compound processes (also known as continuous time random walks) and the equality in law in the space of the trajectories of the stochastic processes. In Section 3 we focus on the non-homogeneous Poisson case, showing the connection to the infinitesimal definition of certain point processes. In the homogeneous case we show that this class of processes is closed with respect to the composition with Bernstein subordinators (which produces the so-called Bernstein fractional counting processes \cite{OT2015}). We also briefly study the behaviour of the time-changing with the inverse of a Bernstein subordinator which modifies the compound Poisson process into a compound renewal and we provide the related governing equation.

After some basics on the integral of a stochastic process, the last section 	of the paper focuses on the integral of compound processes, both Poisson and renewal type. We begin with the analysis of the fractional integral of a compound Poisson process and then extend the result to more general integral forms, proving that for fixed $t$, the integral is distributed as a compound Poisson random variable. We also derive some limit results, the explicit form of the iterated integrals and the equation governing their joint distribution. Then, we include the derivation of the explicit probability law for the integral of Skellam-type processes. At last, we provide some results on the integral of compound renewal processes in the Fourier-Laplace domain.

\section{Preliminaries on compound processes}

Let $(\Omega, \mathcal{F}, \{\mathcal{F}\}_t, P)$ be a filtered probability space on which all the random elements are defined and  $\mathcal{D}[0,\infty)$ the space of the real c\`{a}dl\`{a}g functions endowed with the Skorokhod topology.

We begin by recalling some known results which will be useful throughout the rest of the paper.
\begin{lemma}\label{lemmaUguaglianzaProcessiSpazioFunzioniCadlag}
	Let $X, Y$ be two stochastic processes on $\mathcal{D}[0,\infty)$ starting at $0$ a.s.. Assume that $X(t)\stackrel{d}{=}Y(t), \ t\ge0$. 
	\begin{itemize}
		\item[($i$)] If both $X$ and $Y$ have independent and stationary increments, then $X\stackrel{d}{=}Y$ on $\mathcal{D}[0,\infty)$.
		\item[($ii$)] If both $X$ and $Y$ have independent increments and $X(t)-X(s)\stackrel{d}{=}Y(t)-Y(s), \ t>s\ge0$, then $X\stackrel{d}{=}Y$ on $\mathcal{D}[0,\infty)$.
	\end{itemize} 
\end{lemma}

We point out that the equality in $\mathcal{D}[0,\infty)$ is important to obtain a pointwise equality for the integral process.

The reader can find different definitions of compound processes (both Poisson and non-Poisson) as well as different names, such as continuous time random walks. Here we use the following definition.

\begin{definition}\label{definizioneProcessoComposto}
	Let $N$ be a counting process and $X$ a real r.v. (non degenerate at $0$) with distribution $F$. We define the compound process of $N$ with $X$ as $M=\big\{M(t) = \sum_{k=1}^{N(t)} X_k\big\}_{t\ge0}$, with i.i.d. $X_k \stackrel{d}{=}X$ and independent of $N$. We write $M\sim Compound(N,X)$ or equivalently $M\sim Compound(N,F)$.	
\end{definition}

Note that we assume that $X$ is not degenerate at $0$ in order to avoid the case of an identically null compound process.

If $N$ is a renewal process with waiting times distributed as $W\sim G$, then we call $M$ a compound renewal process, $M\sim CR(W, X)$. 
If $N$ is a (non-homogeneous) Poisson process with rate function $\lambda: [0,\infty) \longrightarrow [0,\infty]$ continuous in $(0,\infty)$ and integrable, denoted by $N\sim PP(\lambda)$, then we call $M$ a (non-homogeneous) compound Poisson process, $M\sim CP(\lambda, X)$ or $M\sim CP(\lambda, F)$, having characteristic function
\begin{equation}\label{trasformataFourierProcessoCompound}
	\mathbb{E}e^{i\gamma \big( M(t)-M(s)\big)} = e^{-\int_s^t \lambda(u)\dif u\big(1- \mathbb{E}e^{i\gamma X}\big) }, \ \ \ t> s \ge0,\ \gamma \in \mathbb{R}.
\end{equation}

Hereafter, $F$ denotes a probability distribution function on $\mathbb{R}$ and, when discussing non-homogeneous Poisson processes, we always consider an arbitrary continuous rate function $\lambda$ such that $\int_0^t\lambda(s)\dif s <\infty, \ t\ge0$.

\begin{lemma}\label{lemmaCompostoIncrementiIndipendenti}
	Let $M\sim Compound(N,X)$.
	\begin{itemize}
		\item[($i$)] If $N$ has independent increments, then $M$ has independent increments.
		\item[($ii$)] If $N$ has stationary increments, then $M$ has stationary increments.
	\end{itemize} 
\end{lemma}

\begin{proof}
	Since $X_k$ are i.i.d. and independent of $N$, then $X_{N(t)+k} \stackrel{d}{=} X,\  k\in\mathbb{N},t\ge0$. Basic calculation proves the statement.
\end{proof}

If $X\in\mathbb{N}$ a.s., $M$ is also a counting process. Then, we have the following result.
\begin{proposition}\label{proposizioneCompostoDiCompostoProcessoConteggio}
	Let $N$ be a counting process, $X$ a natural r.v. and $M\sim Compound(N,X)$. 
	%Let $N$ be a counting process, $X$ a natural r.v. and $Y$ a real r.v.. Let $t\ge0$ and define $M(t) = \sum_{k=1}^{N(t)} X_k$, with i.i.d. $X_k \stackrel{d}{=}X$ and independent of $N$. Then, for i.i.d $Y_{k,h},Y_k\stackrel{d}{=}Y,\ k,h\ge1$, independent of $N$ and $X_1,X_2,\dots$,
	 Let $Y$ be a real r.v. and, for $ k=1,\dots,X$, $Y_k$ be an  independent copy of $Y$ (independent of $N$ and $X$). If $Z\sim Compound(M, Y)$ and $W\sim Compound\big(N, \sum_{k=1}^X Y_k\big)$, then $Z(t)-Z(s) \stackrel{d}{=}W(t)-W(s),\ t>s\ge0$.
\end{proposition}

More explicitly, the statement of Proposition \ref{proposizioneCompostoDiCompostoProcessoConteggio} can be expressed as, $t>s\ge0$,
\begin{equation}\label{uguaglianzaDistribuzioneCompostoDiCompostoProcessoConteggio}
	Z(t)-Z(s) = \sum_{k=M(s)+1}^{M(t)} Y_k \stackrel{d}{=} \sum_{k=N(s)+1}^{N(t)} \sum_{h=1}^{X_k} Y_{k,h} = W(t)-W(s),
\end{equation}
where $Y_{k,h}\stackrel{d}{=}Y,\ k,h\ge1$, are independent of $N$ and $X_1,X_2,\dots$.

\begin{proof}
	Thanks to the independence of the elements we have that $Z(t)-Z(s) \stackrel{d}{=} \sum_{k=1}^{M(t)-M(s)} Y_k$ and, for $\gamma\in\mathbb{R}$,
	\begin{align*}
		\mathbb{E}e^{i\gamma \sum_{k=1}^{M(t)-M(s)} Y_k} &= 	\mathbb{E} \Bigl(	\mathbb{E}e^{i\gamma Y} \Big)^{M(t)-M(s)} \\
		& =\mathbb{E} \Bigl( \mathbb{E}e^{i\gamma Y} \Big)^{\sum_{k=1}^{N(t)-N(s)} X_k} \\
		& = 	\mathbb{E} \Bigg( 	\mathbb{E}\Big( 	\mathbb{E}e^{i\gamma Y}\Big)^X \Bigg)^{N(t)-N(s)}.
	\end{align*}
	The proof completes by observing that for $k\ge1, \ \mathbb{E}e^{i\gamma \sum_{h=1}^{X_k} Y_{k,h} } =\mathbb{E}\Big(\mathbb{E}e^{i\gamma Y}\Big)^X$.
\end{proof}

\begin{remark}\label{remarkUguaglianzaSuSpazioFunzioniContinueCompoundDiCompound}
	Under the hypotheses of Proposition \ref{proposizioneCompostoDiCompostoProcessoConteggio}, this result, together with Lemma \ref{lemmaUguaglianzaProcessiSpazioFunzioniCadlag} and Lemma \ref{lemmaCompostoIncrementiIndipendenti}, implies that if $N$ has independent increments, then $Z$ has independent increments and $Z\stackrel{d}{=}W$ on $\mathcal{D}[0,\infty)$. Furthermore, if $N$ has also stationary increments then $Z$ is a L\'evy process (for point ($ii$) of Lemma \ref{lemmaIncrementiIntegraleProcesso}). \hfill $\diamond$
\end{remark}

\section{Compound Poisson process}

\begin{theorem}
	$M\sim CP(\lambda, F)$ if and only if $M$ is a stochastic process having independent increments, $M(0)=0$ a.s. and, for $h,t\ge0, x\in\mathbb{R}$,
	\begin{align}\label{compoundPoissonInfinitesimale}
		P\{M(t,t+h] \in \dif x\} = \begin{cases}
			\begin{array}{l l}
				\lambda(t)h F(\dif x) + o(h), & x\in\mathbb{R}\setminus\{0\},\\
				1-\lambda(t)h + \lambda(t)h F(\dif 0) + o(h), & x = 0.
			\end{array}
		\end{cases} 
	\end{align}
\end{theorem}

Note that $M(t,t+h] = M(t+h)- M(t)$ and $F(\dif x) = F(x)-F(x^-) = p(t, \dif x)$, i.e. if $F$ is absolutely continuous, $F(\dif x) = f(x)\dif x$, otherwise it represents the mass probability at point $x$. The same holds for $p(t,\dif x) =P\{M(t) \in \dif x\},\ t\ge0,x\in\mathbb{R}$.

\begin{proof}
	Necessity. The non-homogeneous Poisson process $N$ with rate function $\lambda$ has independent increments and $M\sim CP(\lambda, F)$ as well since ($i$) of Lemma \ref{lemmaCompostoIncrementiIndipendenti}. Obviously, $M(0) = 0$. Since $X_1,X_2,\dots \sim F$ are i.i.d. and independent of $N$, we have that
\begin{align}
	P\{M(t,t+h]\in \dif x\} %&= \sum_{n\ge0} P\Big\{\sum_{k=N(t)}^{N[t,t+h)} X_k\in \dif x, N[t, t+h) = n\Big\} \\
	& = \sum_{n=0,1} P\left\{ \sum_{k=1}^0{N(t,t+h]} X_k\in \dif x, N(t,t+h] = n\right\} \nonumber \\
	& = \big(1-\lambda(t)h\big) \delta_0(\dif x) + \lambda(t)h P\{X\in \dif x\} + o(h), \label{dimostrazioneRappresentazioneInfinitesimaleCompoundPoisson}
\end{align}	
with $\delta_0$ being the Dirac measure. Formula \eqref{dimostrazioneRappresentazioneInfinitesimaleCompoundPoisson} coincides with \eqref{compoundPoissonInfinitesimale}.
	
	Sufficiency. By standard arguments we obtain 
% that,	$$p(t + \dif t, \dif x) = \lambda(t) \int_{\mathbb{R}} p(t, \dif x-y)F(\dif y) + p(t, \dif x)\big(1-\lambda(t) \dif t\big)$$
	the governing forward equation
	\begin{align}
		\frac{\dif}{\dif t}p(t, \dif x) &= \lambda(t) \Bigg( \int_{\mathbb{R}} p(t, \dif x-y)F(\dif y) - p(t,\dif x)\Bigg) \nonumber\\
		& = \lambda(t) \big( F \ast p (t, \dif x) - p(t, \dif x) \big),\ \ \ x\in\mathbb{R}, \label{equazioneCompoundPoisson}
	\end{align}
	subject to $p(0, \dif x) = \delta_0(\dif x),\ x\in\mathbb{R}$.
	\\Now, using $x$-Fourier transform and using the notation $\mathcal{F} f(\gamma) = \int_\mathbb{R} e^{i\gamma x}f(x)\dif x,\ \gamma\in\mathbb{R}$, equation \eqref{equazioneCompoundPoisson} turns into
	\begin{equation}
		\frac{\dif}{\dif t}\mathcal{F}p(t, \gamma) = \lambda(t)\Big(\mathbb{E}e^{i\gamma X} -1\Big)\mathcal{F}p(t,\gamma),\ \ \ \mathcal{F}p(0,\gamma) = 1,\ \ \ \gamma \in\mathbb{R},
	\end{equation}
	and \eqref{trasformataFourierProcessoCompound} with $s=0$ solves it. The proof concludes by using point ($ii$) of Lemma \ref{lemmaUguaglianzaProcessiSpazioFunzioniCadlag} and the independence of the increments, which implies that $M(t)-M(s)$ has characteristic function \eqref{trasformataFourierProcessoCompound}.
\end{proof}

\begin{remark}\label{remarkEquazioneOperatorialePoissonComposto}
	The proof above also shows that the probability law of $M\sim CP(\lambda, F)$ satisfies equation \eqref{equazioneCompoundPoisson}, which can equivalently be written as
	\begin{equation}\label{equazioneOperatorialeCompoundPoisson}
		\frac{\dif}{\dif t}p(t, \dif x) = -\lambda(t) \big( I- C_F \big) p (t, \dif x),\ \ \ t\ge0,\ x\in\mathbb{R},
	\end{equation}
	where $I$ is the identity operator and $C_F g(t, x) = F(t,x)\ast g = \int_\mathbb{R} g(t, x-y) F(\dif y)$ denotes the convolution operator with respect to the distribution $F$. In the homogeneous case, we are using this formulation to study the effect of applying a Bernstein function to the operator on the right-hand side of \eqref{equazioneOperatorialeCompoundPoisson}, see Theorem \ref{teoremaComposizioneCompoundPoissonConSubordinatore}.
	
	We observe that \eqref{equazioneOperatorialeCompoundPoisson} also holds for any Borel set $A \in\mathcal{B}(\mathbb{R})$ as shown by a direct check. Indeed, with $t\ge0$ and denoting by $\Lambda(t) = \int_0^t \lambda(s)\dif s$ and $F^{(n)} = \conv_{k=1}^n F$, 
	\begin{align}
		\frac{\partial }{\partial t} p(t, A) &= \sum_{n\ge0} \frac{F^{(n)}(A)}{n!} \frac{\partial }{\partial t} e^{-\Lambda (t)} \Lambda(t)^n\nonumber \\
		& = \sum_{n\ge0} \frac{F^{(n)}(A)}{n!} \lambda(t)e^{-\Lambda(t)}\left(-\Lambda(t)^n + n\Lambda(t)^{n-1} \right) \nonumber\\
		& = -\lambda(t)p(t, A) + \lambda(t)F\ast p(t, A). \label{equazioneGeneralePoissonCompostoBoreliani}
		\end{align}
	Clearly, the initial condition is $p(0,A) = \delta_0(A)$ with $A \in\mathcal{B}(\mathbb{R})$.
	\\
	Fix $A \in\mathcal{B}(\mathbb{R})$, the differential equation \eqref{equazioneGeneralePoissonCompostoBoreliani} shows that 
	\begin{equation}\label{finitezzaDerivataLeggePoissonComposto}
		 \Big|\frac{\partial }{\partial t} p(t, A)\Big| = \lambda(t)p(t, A) + \lambda(t)F\ast p(t, A) \le 2\lambda(t) <\infty,\ \ \ t>0, 
	\end{equation}
	where the last inequality comes from the continuity of $\lambda$ in $(0,\infty)$. Therefore, $p(\cdot, A)\in Lip_{loc}(0,\infty)$ since it has bounded derivative over the compact sets.
	\hfill$\diamond$
\end{remark}

\begin{example}[Skellam-type processes]\label{esempioSkellam}
	Let $\mathcal{I}\subset \mathbb{R}\setminus\{0\}$ be a finite set and $M\sim CP(\lambda, F)$ with distribution $F(x) = \sum_{i\in \mathcal{I}} a_i \mathds{1}(x\ge i)$, such that $a_i>0\ \forall\ i,\ \sum_{i\in\mathcal{I}}a_i = 1$. Thus, $F(\dif x)= \sum_{i\in\mathcal{I}} a_i \delta_i(\dif x)$, where $\delta_i$ is the Dirac delta measure centred in $i$. $M$ is a discrete point process which jumps, at Poisson times, with size $i\in \mathcal{I}$ with probability $a_i$. This is a Skellam-type process, see for instance \cite{CO202} and references therein; if $\mathcal{I} = \{-K,\dots,-1,1,\dots,K\}$, with $K\in\mathbb{N}$, $M$ is called a Skellam process of order $K$; if $\mathcal{I}\subset(0,\infty)$, it is called generalized counting process, see for instance \cite{DcMM2016}; if $\mathcal{I}=\{1\}$, it is a Poisson process. 
	
	Hence, if $M\sim CP(\lambda, F)$ with $F$ defined above and  $N_i\sim PP(a_i\lambda) = CP(a_i\lambda, 1), \ i\in\mathcal{I},$ are independent Poisson processes, then  $M\stackrel{d}{=} \sum_{i \in\mathcal{I} } i N_i.$
	
	From \eqref{equazioneCompoundPoisson} we easily derive the differential equation governing the probability law of $M$. For $x\in\mathbb{Z}$,
	\begin{align}
			\frac{\dif}{\dif t}p(t, x) &= \lambda(t) \Bigg( \int_{\mathbb{R}} p(t, x-y)\sum_{i\in \mathcal{I}} a_i\delta_i(\dif y) - p(t, x)\Bigg)\nonumber\\
			& =\lambda(t)\sum_{i\in\mathcal{I}}a_i \big(p(t, x-i) -p(t,x)\big),\ \ \ t\ge0,
	\end{align}
	which is a particular case of formula (2.4) of \cite{CO202} (with $\lambda_i = a_i\lambda$). \hfill$\diamond$
\end{example}

\subsection{Homogeneous case}

We recall that $f:[0,\infty)\longrightarrow [0,\infty)$ is a Bernstein function if $f\in C^\infty, \ (-1)^n \dif^n f/\dif x^n \le 0, \ n\ge1$, and it can be expressed as
\begin{equation}\label{definizioneFunzioneBernstein}
	f(x) = a+bx+ \int_0^\infty\Bigl( 1-e^{-xw}\Bigr) \bar \nu_f(\dif w), \ \ \ x\ge0,
\end{equation}
where $a,b\ge0$ and $\bar \nu_f$ is a L\'{e}vy measure, that is such that $\int_0^\infty (s\wedge1)\bar \nu_f(\dif s)<\infty$. We refer to \cite{SSV2012} for a detailed discussion.
Bernstein functions are related to non-decreasing L\'{e}vy processes, also known as subordinators. Indeed, for each Bernstein function $f$ there exists a subordinator $\mathcal{H}_f$ such that $f$ is the L\'{e}vy symbol of $\mathcal{H}_f$, i.e. $\mathbb{E}e^{-\mu\mathcal{H}_f(t)} = e^{-tf(\mu)}, \ \mu, t\ge0$ % Hereafter we assume $a=b=0$
(and, as is well known, $\mathcal{H}_f$ has independent and stationary increments).

\begin{lemma}\label{lemmaEquazioneOperatorialeBernstein}
	Let $\lambda>0$, $f$ be a Bernstein function as in \eqref{definizioneFunzioneBernstein} with $ \bar \nu_f\not=0$ and
	\begin{equation}\label{pesiBernstein}
		a_k = a_{\lambda,f,k} = \frac{1}{f(\lambda)} \frac{\lambda^k}{k!}\int_0^\infty e^{-\lambda w} w^k \bar \nu_f(\dif w),\ \ \ k\ge1.
	\end{equation}
	\begin{itemize}
		\item[($i$)] $\displaystyle \sum_{k\ge1} a_k = 1-(a+b\lambda)/f(\lambda)$.
		\item[($ii$)] $ \displaystyle f\big(\lambda(1-x)\big) = f(\lambda) \Bigg(1- \sum_{k\ge 1}a_k x^k\Bigg) - b\lambda x ,\ \ \ x\le1. $
		\item[($iii$)] Let  $a'_1 = a_1 + b\lambda/f(\lambda)$ and $a'_k = a_k,\ k\ge2$. Let $F$ be a distribution function and $p$ be a suitable function. $f$ has $a=0$ if and only if
		\begin{equation}\label{relazioneOperatorialeFunzioneBernstein}
			f\Big(\lambda\big(I-C_F\big)\Big) p= f(\lambda)\big(I-C_{ F_f}\big) p
		\end{equation}
		with $F_f$ being the distribution
		\begin{equation}\label{distribuzioneComposizioneBernstein}
		%	F_f(x) = \frac{1}{f(\lambda)}\sum_{k\ge1} \frac{\lambda^k}{k!}\int_0^\infty e^{-\lambda w} w^k \bar \nu_f_f(\dif w) \, F^{(k)}(x), \ \ \ x\in \mathbb{R}
		F_f = \sum_{k\ge1} a'_k F^{(k)}, \ \ \ x\in \mathbb{R},
		\end{equation}
		where $F^{(k)} = \conv_{h=1}^k F$ is the $k$-fold convolution of $F$.
	\end{itemize} 
\end{lemma}

%Hereafter we denote the weights of the combination in \eqref{distribuzioneComposizioneBernstein} by $a_k,\ k\ge1$, i.e. $ F_f = \sum_{k \ge1} a_k F^{(k)}$.

\begin{proof}
	($i$) For any $a,b\ge0$, we have
	\begin{align}
		\sum_{k\ge1} a_k % \frac{1}{f(\lambda)}\sum_{k\ge1} \frac{\lambda^k}{k!}\int_0^\infty e^{-\lambda w} w^k \bar \nu_f(\dif w) 
		&= \frac{1}{f(\lambda)}\int_0^\infty e^{-\lambda w}\sum_{k\ge1} \frac{(\lambda w)^k}{k!} \bar \nu_f(\dif w)\nonumber \\
		& =\frac{1}{f(\lambda)} \int_0^\infty  e^{-\lambda w} (e^{\lambda w}-1)\bar \nu_f(\dif w) \nonumber\\
		& = \frac{f(\lambda)-a- b\lambda}{f(\lambda)}.\label{calcoliPesiBernstainSommano1}
	\end{align}
	($ii$) For any $a,b\ge0$ and $x\le 1$, we have
	\begin{align}
		f\big(\lambda(1-x)\big) & = a+b\lambda(1-x) +\int_0^\infty \Big(1-e^{-\lambda w} \sum_{k\ge0}\frac{(\lambda x w)^k}{k!} \Big)\bar \nu_f(\dif w)\nonumber\\
		& =  a + b\lambda(1-x) + \int_0^\infty (1-e^{-\lambda w})\bar \nu_f (\dif w) - \sum_{k\ge1}\frac{(\lambda x )^k}{k!}  \int_0^\infty e^{-\lambda w} w^k \bar \nu_f(\dif w)\nonumber\\
		& = f(\lambda)\Big(1- \frac{1}{f(\lambda)}\sum_{k\ge1}\frac{(\lambda x )^k}{k!}  \int_0^\infty e^{-\lambda w} w^k \bar \nu_f(\dif w)\Big) -b\lambda x.
	\end{align}
	($iii$) $F_f$ is a distribution since it is a convex combination of the distributions $F^{(k)},\ k\ge1$. Indeed, for point ($i$) we have $\sum_{k} a'_k = \sum_{k}a_k+\lambda b/f(\lambda) = 1-a/f(\lambda)$ which is $1$ if and only if $a=0$.
	Now, by suitably adapting the steps of point ($ii$) we obtain
	\begin{align}
		f\Big(\lambda\big(I-C_F\big)\Big) p& =  f(\lambda)\Big(I- \frac{1}{f(\lambda)}\sum_{k\ge1}\frac{\lambda ^k}{k!}  \int_0^\infty e^{-\lambda w} w^k \bar \nu_f(\dif w)C_F^k\Big)p - b\lambda C_Fp \nonumber \\
		& =  f(\lambda)\Bigg(p- \Big(a_1+\frac{\lambda b}{f(\lambda)}\Big)F\ast p - \sum_{k\ge2}a_k \Big(\conv_{h=1}^k F\Big)\ast p \Bigg) \nonumber \\
		& = f(\lambda) \Big(p- \sum_{k\ge1} a'_k F^{(k)}\ast p \Big)\nonumber\\
		& =  f(\lambda) \Big(I- C_{F_f} \Big)p \nonumber
	\end{align}
	which proves \eqref{relazioneOperatorialeFunzioneBernstein}.
\end{proof}

%\begin{theorem}\label{teoremaEquazioneDifferenzialeBernsteinCompound}
%	Let $\lambda>0$, $F$ be a distribution function and $f$ be a Bernstein function with $a=b=0, \bar \nu_f\not=0$. Then, the solution to 
%		\begin{equation}\label{equazioneDifferenzialeOperatorialeFunzioneBernstein}
%		\frac{\dif}{\dif t}p(t, \dif x) = - f\Big(\lambda\big(I-C_F\big)\Big)p(t, \dif x) 
%	\end{equation}
%	is the probability mass function of $M_f\sim CP\big(f(\lambda), F_f\big)$, with $F_f$ in \eqref{distribuzioneComposizioneBernstein}, and $M_f = M\circ\mathcal{H}_f$. Finally, $M_f$ has jumps distributed as
%	\begin{equation}\label{leggeSaltiCompoundBernstein}
%		Y = \sum_{k\ge1} \mathds{1}(B = k) \sum_{h=1}^k X_h,\ \ \text{such that}\ \ P\{B = k\} = a_k,\ k\ge1,
%	\end{equation}
%	with $B$ and $X_1,X_2,\dots\sim F$ independent.
%\end{theorem}
%
%\begin{proof}
%	
%	The jumps \eqref{leggeSaltiCompoundBernstein} follow from the structure of the distribution $F_f$ in \eqref{distribuzioneComposizioneBernstein}.
%\end{proof}

\begin{theorem}\label{teoremaComposizioneCompoundPoissonConSubordinatore}
Let $M\sim CP(\lambda, F)$, $f$ be a Bernstein function (see \eqref{definizioneFunzioneBernstein}) and $\mathcal{H}_f$ be the corresponding Bernstein subordinator, independent of $M$.
\begin{itemize}
	\item[($i$)] $a=0$ if and only if $M\circ \mathcal{H}_f\sim CP\big(f(\lambda), F_f\big)$, with %$F_f = F$ if $b\not=0$ and 
	$F_f$ in \eqref{distribuzioneComposizioneBernstein}. The jumps are distributed as $Y = \sum_{h=1}^A X_h$, with $A$ being independent of $X_1,X_2,\dots\sim F$ and $M$ and $P\{A = k\} = a'_k,\ k\ge1,$ with $a'_k$ in ($iii$) of Lemma \ref{lemmaEquazioneOperatorialeBernstein}.
	%\begin{equation}\label{leggeSaltiCompoundBernstein}
	%Y = \sum_{k\ge1} \mathds{1}(B = k) \sum_{h=1}^k X_h,\ \ \text{such that}\ \ P\{B = k\} = a_k,\ k\ge1,
	%\end{equation}
	\item[($ii$)] Let $a=0$. The probability law of $M\circ\mathcal{H}_f$ satisfies, for $x\in\mathbb{R}$,	\begin{equation}\label{equazioneDifferenzialeOperatorialeFunzioneBernstein}
		\frac{\dif}{\dif t}p(t, \dif x) = - f\Big(\lambda\big(I-C_F\big)\Big)p(t, \dif x), \ t\ge0,\ \ \ p(0,\dif x) = \delta_0(\dif x).
	\end{equation}
\end{itemize}
\end{theorem}

Note that statement ($i$) would actually need the hypothesis $f\not=0$ in order to avoid the degenerate case in which $f$ is identically null (which is of no interest and therefore omitted).

\begin{proof}
	($i$) We use the joint characteristic function of the increments of $M\circ \mathcal{H}_f$. Let $n\in\mathbb{N}, \gamma_1,\dots,\gamma_n \in \mathbb{R}$ and $0=t_0<t_1<\dots< t_n$, 
	\begin{align}
		\mathbb{E}& \exp\biggl(i\sum_{j=1}^n \gamma_j \Big( M\circ\mathcal{H}_f(t_{j}) - M\circ \mathcal{H}_f(t_{j-1})\Big)\bigg)\nonumber \\
		&= \mathbb{E}\Bigg[\mathbb{E}\bigg[\prod_{j=1}^n \exp\bigg(i\gamma_j \sum_{k=1}^{N\circ\mathcal{H}_f(t_{j})- N\circ\mathcal{H}_f(t_{j-1})} X_k  \bigg)\bigg| \mathcal{H}_f(s),\,s\in[0,t_n]\bigg]\Bigg] \nonumber\\
		& = \mathbb{E}\Bigg[\prod_{j=1}^n \mathbb{E}\bigg[ \exp\bigg(i\gamma_j \sum_{k=1}^{N\circ\mathcal{H}_f(t_{j})- N\circ\mathcal{H}_f(t_{j-1})} X_k  \bigg)\bigg| \mathcal{H}_f(s),\,s\in[0,t_n]\bigg]\Bigg] \nonumber\\
		& =  \mathbb{E}\prod_{j=1}^n \exp \bigg(-\lambda \big(\mathcal{H}_f(t_j)- \mathcal{H}_f(t_{j-1})\big) \Big(1-\mathbb{E}e^{i\gamma_j X}\Big) \bigg)\label{passaggioUtilePerDimostrazioneNecessitaSubordinazione}\\
		& = \prod_{j=1}^n \exp \bigg(- (t_j -t_{j-1}) f\Big(\lambda\big(1-\mathbb{E}e^{i\gamma_j X}\big)\Big) \bigg),
		\label{trasformataIncrementiComposizioneCompoundPoissonSubordinatore}
	\end{align}
	where in the first step we used that the $X$s are i.i.d. and independent of all the other elements; in the second step we used that $\mathcal{H}_f$ is non-decreasing and $M$ has independent increments; in the third step we used that $M\sim CP(\lambda, X)$ independent of $\mathcal{H}_f$.
	Formula \eqref{trasformataIncrementiComposizioneCompoundPoissonSubordinatore} proves that $M\circ \mathcal{H}_f$ has independent and stationary increments. Note that the above calculation could have been avoided by using classical results on subordination of Lèvy processes, see for instance Theorem 1.3.25 of \cite{A2009}, but they are going to be useful in the proof of the next statement.
	Now, it suffices to prove that $M\circ \mathcal{H}_f$ is pointwise compound Poisson. For $t\ge0$,
	\begin{align}
		\mathbb{E} e^{i\gamma M\big(\mathcal{H}_f(t)\big)} &= \exp\Bigg(-t f\Big(\lambda\big(1-\mathbb{E}e^{i\gamma X}\big) \Big) \Bigg) \nonumber\\
		& = \exp \Bigg(-tf(\lambda)\Big(1-\sum_{k\ge1} a_k\big(\mathbb{E}e^{i\gamma X}\big)^k \Big) - b\lambda \mathbb{E}e^{i\gamma X} \Bigg) \nonumber\\
		& =  \exp \Bigg(-tf(\lambda)\Big(1-\sum_{k\ge1} a'_k \mathbb{E}e^{i\gamma \sum_{h=1}^k X_h}\Big) \Bigg),
	\end{align}
	with $X_1,\dots $ being independent copies of $X$. In the second step we used ($ii$) of Lemma \ref{lemmaEquazioneOperatorialeBernstein}. Now, since $\sum_{k} a'_k = 1$ if and only if $a=0$ we obtain that $\sum_{k\ge1} a'_k  \mathbb{E}\exp\big(i\gamma \sum_{h=1}^k X_h\big) = \mathbb{E}\exp\big(i\gamma \sum_{h=1}^A X_h\big)$ with $A$ independent r.v. such that $P\{A = k\}=a'_k,\ k\ge1$. Thus, $M\circ\mathcal{H}_f\sim CP\big(f(\lambda), F_f\big)$, where $F_f$ derives from the structure of the jumps, see also \eqref{distribuzioneComposizioneBernstein}.
	\\
	($ii$) From point ($i$) and \eqref{equazioneOperatorialeCompoundPoisson}, $x\in\mathbb{R}$,
	\begin{equation*}
		\frac{\dif}{\dif t}p(t, \dif x) = - f(\lambda)\big(I-C_{F_f}\big)p(t, \dif x), \ t\ge0,\ \ \ p(0,\dif x) = \delta_0(\dif x).
	\end{equation*} 
	The statement follows from point ($iii$) of Lemma \ref{lemmaEquazioneOperatorialeBernstein}.	
\end{proof}

\begin{theorem}\label{teoremaComposizioneCompoundPoissonGenerale}
	Let $M\sim CP(\lambda, F)$ and $T$ be an independent non-decreasing c\`{a}dl\`{a}g stochastic process such that $T(0) = 0$ a.s.. $M\circ T\sim CP$ if and only if $T$ is a Bernstein subordinator ($\mathcal{H}_f$) with $f$ having $a=0$. 
\end{theorem}

Note that Definition \ref{definizioneProcessoComposto} excludes the degenerate case of $M$ identically null (i.e. having jumps of the form $X\stackrel{a.s.}{=}0$).
\\

Before proving Theorem \ref{teoremaComposizioneCompoundPoissonGenerale} we display the following useful result.

\begin{lemma}\label{lemmaEstensioneUguaglianzaLaplaceComplessa}
	Let $X_1,\dots, X_n$ be non-negative (real) random variables. Let $H = \{z\in \mathbb{C}\,:\, \Re(z)>0\}$ and $E_1,\dots,E_n\subset H$ each one having an accumulation point in $H$. If 
	\begin{equation}\label{ipotesiLemmaEstensioneUguaglianzaTrasformataLaplaceComplessa}
		\mathbb{E}\prod_{j=1}^n e^{z_jX_j} = \prod_{j=1}^n \mathbb{E} e^{z_jX_j},\ \ \ (z_1,\dots,z_n)\in E_1\times\dots\times E_n,
	\end{equation}
	then $X_1,\dots, X_n$ are independent.
\end{lemma}

\begin{proof}
	The functions $L_{X_j}(z) = \mathbb{E}e^{zX_j}, \ j=1,\dots, n,\ z\in\mathbb{C}$, are holomorphic in $H$ and $L(z_1,\dots, z_n) = \mathbb{E}\prod_{j=1}^n e^{z_jX_j}$ is holomorphic in $H$ with respect to each variable (keeping the others fixed; it is also holomorphic in $H^n$, but this is not needed). Thus, also the function $D(z_1,\dots,z_n) = L(z_1,\dots, z_n)-\prod_{j=1}^n L_{X_j}(z_j)$ is holomorphic in $H$ with respect to each variable. \\
	By the hypothesis \eqref{ipotesiLemmaEstensioneUguaglianzaTrasformataLaplaceComplessa} we show that $D(z_1,\dots, z_n)=0$ for $z_1,\dots,z_n\in H$. Fix $(z_2,\dots,z_n)\in E_2\times\dots\times E_n$ and let us define $G_1(z) = D(z,z_2,\dots,z_n)$. $G_1$ is holomorphic in $H$ and, by \eqref{ipotesiLemmaEstensioneUguaglianzaTrasformataLaplaceComplessa}, $G_1(z) = 0$ for $z\in E_1$ which has an accumulation point in $H$, therefore, for the identity theorem of the holomorphic function it must be $G_1 (z) = 0$ for $z\in H$. By iterating the same argument for each variable of $D$ we obtain that $D=0$ on $H^n$. Thus, the equality \eqref{ipotesiLemmaEstensioneUguaglianzaTrasformataLaplaceComplessa} holds on $H^n$ and therefore on the set $(0,\infty)^n$ meaning that $X_1,\dots,X_n$ are independent by the property of the Laplace transform.
\end{proof}

%We can now prove Theorem \ref{teoremaComposizioneCompoundPoissonGenerale}
\begin{proof}[Proof of Theorem \ref{teoremaComposizioneCompoundPoissonGenerale}]
	Sufficiency. If $T=\mathcal{H}_f$ and $a=0$, then $M\circ\mathcal{H}_f\sim CP$ as proved in point ($i$) of Theorem \ref{teoremaComposizioneCompoundPoissonConSubordinatore}.
	\\
%	Necessity. We first show that $T$ has stationary increments. Let $\gamma\in\mathbb{R}$,
%	\begin{equation}
%		\mathbb{E} e^{i\gamma \big(M(T(t))-M(T(s))\big)} = \mathbb{E} \exp\Big(-\lambda \big(T(t)-T(s)\big) \big(1-\mathbb{E}e^{i\gamma X}\big) \Big).\nonumber
%	\end{equation}
%	By recalling form \eqref{trasformataFourierProcessoCompound} of the characteristic function of a compound Poisson process, $T$ must have stationary increments and exponential characteristic function.
%	
	Necessity. We prove that $T$ has stationary and independent increments. Let $n\in \mathbb{N}$, $\gamma_1,\dots,\gamma_n\in\mathbb{R}$ and $0=t_0<t_1<\dots<t_n$. The joint characteristic function of the increments of $M\circ T$, by following the steps leading to \eqref{passaggioUtilePerDimostrazioneNecessitaSubordinazione}, reads
	\begin{align}
 		\mathbb{E} \prod_{j=1}^n e^{i\gamma_j \big(M(T(t_j))-M(T(t_{j-1}))\big)} &= \mathbb{E}\prod_{j=1}^n \exp \bigg(-\lambda \big(T(t_j) - T(t_{j-1})\big) \Big(1-\mathbb{E}e^{i\gamma_j X}\Big) \bigg)\nonumber\\
 		& = \mathbb{E}\prod_{j=1}^n e^{-\big(T(t_j) - T(t_{j-1})\big) g(\gamma_j)},\label{caratteristicaIncrementiComposizioneGeneralePoissonCompostoI}
 	\end{align}
 	with $g(\gamma) = \lambda(1-\mathbb{E}e^{i\gamma X})\ge 0,\  \gamma\in \mathbb{R}$. With \eqref{caratteristicaIncrementiComposizioneGeneralePoissonCompostoI} at hand for $n=1$, by recalling the form \eqref{trasformataFourierProcessoCompound} of the characteristic function of a compound Poisson process, $T$ must have stationary increments.
 	\\
 	To prove the independence of the increments, by recalling that $M\circ T\sim CP$, i.e. with independent increments, we obtain that
 	\begin{align}
 		\mathbb{E} \prod_{j=1}^n e^{i\gamma_j \big(M(T(t_j))-M(T(t_{j-1}))\big)} &= \prod_{j=1}^n \mathbb{E} e^{i\gamma_j \big(M(T(t_j))-M(T(t_{j-1}))\big)}\nonumber\\
 		& = \prod_{j=1}^n \mathbb{E} e^{ -\big(T(t_j) - T(t_{j-1})\big) g(\gamma_j)} , \label{caratteristicaIncrementiComposizioneGeneralePoissonCompostoII}
 	\end{align}
	 where in the last step we used \eqref{caratteristicaIncrementiComposizioneGeneralePoissonCompostoI} for $n=1$, for each $j$. Therefore, by putting together \eqref{caratteristicaIncrementiComposizioneGeneralePoissonCompostoI} and \eqref{caratteristicaIncrementiComposizioneGeneralePoissonCompostoII} we obtain that
	 \begin{align}\label{equazioneIncrementiSuSottoinsieme}
	  	\mathbb{E}\prod_{j=1}^n e^{ -\big(T(t_j) - T(t_{j-1})\big) \beta_j }= \prod_{j=1}^n \mathbb{E} e^{ -\big(T(t_j) - T(t_{j-1})\big) \beta_j },
	 \end{align}
	 for $\beta_j = g(\gamma_j) = \lambda(1-\mathbb{E}e^{i\gamma_j X})\ge0,\ \gamma_j\in\mathbb{R}$ and $j=1,\dots,n$. This means that\eqref{equazioneIncrementiSuSottoinsieme} holds for $\beta_1,\dots,\beta_n\in g(\mathbb{R})\subset \{z\in\mathbb{C}\,:\, \Re(z)\ge0	\}$ (which is not the complex halfplane with strictly positive real part $H$ needed to apply Lemma \ref{lemmaEstensioneUguaglianzaLaplaceComplessa}). 
	 \\
	 Now, note that $g(0)=0$ is the only element of $g(\mathbb{R})$ having real part equal to $0$; in this case $\mathbb{E}e^{i\gamma X}=1 +i b =:z$ and since $|\mathbb{E}e^{i\gamma X}|\le 1 $ is inside the unitary disc $C_1$, the only complex point $z\in C_1$ such that $\Re(z)=1$ is $z=1+i0 = \mathbb{E}e^{i0X}$. 
	 \\
	 Thus, $g(\mathbb{R}) \subset H\cup \{0\}$. Furthermore, $g$ is not constant since $X$ is non-degenerate in $0$, thus $g(\mathbb{R})\cap H\not=\emptyset$. Finally, by observing that $g$ is continuous (any characteristic function is continuous) every point of $g(\mathbb{R})$ is an accumulation point and therefore $g(\mathbb{R})$ admits an accumulation point in $H$ and we can use Lemma \ref{lemmaEstensioneUguaglianzaLaplaceComplessa} obtaining that the increments of $T$ are independent.
	 
	 Hence, $T$ is a  c\`{a}dl\`{a}g process starting from $0$ a.s. and having stationary and independent increments, meaning that it must be a Bernstein subordinator and must be $a=0$ because of point ($i$) of Theorem \ref{teoremaComposizioneCompoundPoissonConSubordinatore}.
\end{proof}

\begin{example}[Space-fractional Poisson process]
	As an example, we show that if $M$ is a homogeneous Poisson process and $\mathcal{H}_f$ is a Bernstein subordinator with $a=b=0$, Theorem \ref{teoremaComposizioneCompoundPoissonConSubordinatore} easily yields the process defined in equation (1.1) of \cite{OT2015}. Indeed, $M\sim PP(\lambda) = CP(\lambda, 1)$, then $F(\dif x) = \delta_{1}(\dif x)$ and, from \eqref{distribuzioneComposizioneBernstein}, $F_f(\dif x) = \sum_{k\ge1} a'_k \sconv_{h=1}^k F(\dif x) = \sum_{k\ge1} a_k \delta_k(\dif x)$ (note that $a'_k = a_k\ \forall\ k$ since $b=0$, see point ($iii$) of Lemma \ref{lemmaEquazioneOperatorialeBernstein}). Therefore, Theorem \ref{teoremaComposizioneCompoundPoissonConSubordinatore} implies that $N\circ \mathcal{H}_f \sim CP\big(f(\lambda), F_f\big)$, that is a process having infinitesimal behavior (see \eqref{compoundPoissonInfinitesimale}), for $t\ge0$,
	\begin{align*}
		P\{M(t,t+\dif t]  = x\} = \begin{cases}
			\begin{array}{l l}
				f(\lambda)\dif t \,a_x + o(\dif t), & x\in\mathbb{N},\\
				1-f(\lambda)\dif t + o(\dif t), & x = 0,\\
				o(\dif t), & otherwise,
			\end{array}
		\end{cases}
	\end{align*}
	which coincides with (1.1) of \cite{OT2015} by using \eqref{pesiBernstein}. \hfill$\diamond$
\end{example}

We now study the time-changed compound Poisson process with the inverse $L_f$ of the Bernstein subordinator $\mathcal{H}_f$, i.e. $L_f(t)=\inf\{s\ge0\,:\, \mathcal{H}_f(s)\ge t\},\ t\ge0$.
%It is well-known that 
%\begin{equation}
%	\int_0^\infty e^{-\mu t} l_f(t,s)\dif t = -\frac{1}{\mu}  \frac{\partial }{\partial s} \mathbb{E} e^{-\mu \mathcal{H}_f(s)}= \frac{f(\mu)}{\mu} e^{-s f(\mu)},\ \ \ \mu>0.
%\end{equation}

We recall (see \cite{K2011}) that if $N\sim PP(\lambda)$ and $L_f$ is an independent inverse subordinator, then, $N\circ L_f$ is a renewal process with waiting times $W$ s.t. 
	\begin{equation}\label{intertempoComposizionePoissonInversoSubordinatore}
		P\{W\le t\} =  1-\mathbb{E}e^{-\lambda L_f(t)}\ \text{ and } \ \mathbb{E}e^{-\mu W} = \frac{\lambda}{\lambda+f(\mu)},\ \mu, t>0.
	\end{equation}
Thus, if $M\sim CP(\lambda, X)$, then $M\circ L_f\sim CR(W, X)$ with $W$ distributed as described in \eqref{intertempoComposizionePoissonInversoSubordinatore}.

%\begin{proof}
%	We begin by showing that the first intertime ($W_1$) of the time-changed process is distributed as the $X$ in the statement. For $t\ge0$,
%	\begin{align}
%		P\{W_1 < t\} = 1-P\{N\circ L_f(t) = 0\} = 1-\mathbb{E}e^{-\lambda L_f(t)},
%	\end{align}
%	and its Laplace transform is
%	\begin{align}
%		\mathbb{E}e^{-\mu W_1} = 
%	\end{align}
%\end{proof}

Now, we show the equation governing the probability law of the compound renewal process $M\circ L_f$, which requires the convolutional-type derivative introduced by different authors, see \cite{K2011, T2015}, 
\begin{equation}\label{definizioneDervitaConvoluzionale}
	\mathcal{D}_t^f u(t) = b\frac{\dif }{\dif t}u(t) + \int_0^t \frac{\partial }{\partial t}u(t-s)\nu_f(s)\dif s,
\end{equation}
where $\nu_f$ is the tail of the L\'{e}vy measure $\bar \nu_f$. We refer to Section 2 of \cite{BS2024} (or Section 4 of \cite{CO2026}) for some information on the operator in \eqref{definizioneDervitaConvoluzionale} and to \cite{K2011, T2015} for further details.
\\

Hereafter we assume the following condition.
\\
\textbf{Condition I.} The L\'{e}vy measure $\bar\nu_f$ associated to the Bernstein function $f$ is such that $\bar\nu_f(0,\infty) = \infty$ and its tail $\nu_f(s) = a+\bar \nu_f(s,\infty)$ is absolutely continuous.
\\

It was shown in \cite{T2015} that under Condition I, the inverse $L_f$ admits a probability density $l_f(t,x) = P\{L_f(t)\in\dif x\}/\dif x, \ x,t \ge0,$ with $l_f(0,x) = \delta_0(x)$.

In order to prove the next statement we also recall that in \cite{T2015} the author introduced another convolution-type derivative generalizing the Riemann-Liouville fractional derivative, 
\begin{align}
	\mathbb{D}_t^f u(t) =  b\frac{\dif }{\dif t}u(t) + \frac{\dif }{\dif t}\int_0^t u(t-s)\nu_f(s)\dif s,
\end{align}
which relates to the convolution-type derivative \eqref{definizioneDervitaConvoluzionale} by means of (see Proposition 2.7 of \cite{T2015})
\begin{align}\label{relazioneOperatoriConvoluzione}
	\mathcal{D}_t^f u(t) =\mathbb{D}_t^f u(t) - \nu_f(t) u(0).
\end{align}
Furthermore, it is useful to recall (see Theorem 4.1 of \cite{T2015}) that the probability density $l_f$ satisfies the problem, with  $s>0$ if $b = 0$ or $0<s<t/b$ if $b>0$,
\begin{align}\label{legameDerivataConvoluzionaleLiouvilleOperatoreSpaziale}
	\mathbb{D}_t^f l_f(t,s)= -\frac{\partial }{\partial s}l_f(t,s), \ \ \ t\ge0,
\end{align}
subject to $l_f(t,0) = \nu_f(t), \ l_f(t, t/b)=0$ and $l_f(0,s) = \delta_0(s)$.

\begin{proposition}[Time-fractional compound Poisson process]\label{proposizioneComposizioneInversoSubordinatore}
	Let $M\sim CP(\lambda, F)$ and $L_f$ be an inverse Bernstein subordinator satisfying  Condition I and having density $l_f$ such that:
	\begin{align}
		&\forall \ 0<a<c,\   \ \left| \frac{\partial l_f}{\partial t} (t, \cdot)\right|\le g_{a,c}(\cdot)\in L^1(0,\infty), \ \ t\in [a, c], \label{ipotesiI}\\
		&\exists \ \delta>0\ \ s.t.\ \  \int_0^\delta \dif t \int_0^\infty (1 \wedge s ) \left| \frac{\partial l_f}{\partial t} (t, s)\right| \dif s < \infty, \label{ipotesiII}\\
		& \forall \ t>0,\ \ l_f(t,\cdot)\in AC_{loc}(0,\infty),\  \ \int_0^\infty \left| \frac{\partial l_f}{\partial s} (t, s)\right| \dif s <\infty.\label{ipotesiIII}
	\end{align}  
	Then, the  probability function $p_f$ of $M\circ L_f$ satisfies
	\begin{equation}
		\mathcal{D}_t^f p_f(t,A) = -\lambda(I-C_F) p_f(t,A),\ \ \ t>0,\ A\in \mathcal{B}(\mathbb{R}),
	\end{equation}
	subject to $p_f(0,A) = \delta_0(A)$.
\end{proposition}

\begin{proof}
	For the sake of simplicity of the notation, in this proof we omit the dependence on the Bernstein function for $l_f, \nu_f$ and $\bar \nu_f$. Furthermore, we denote by $l'$ and $ p_f'$ the derivatives of $l$ and $p_f$ with respect to the first variable (representing the time).\\
	
	It is easy to evaluate the initial condition, $p_f(0,A) = \int_0^\infty p(s, A)\delta_0(s)\dif s =  p(0,A) = \delta_0(A)$.

	We prove that hypotheses \eqref{ipotesiI} and \eqref{ipotesiII} imply that
	\begin{equation}\label{scambiabilitaOperatoreConvoluzionaleEIntegrale}
		\mathcal{D}_t^f p_f(t,A) = \int_0^\infty \big(p(s,A)- p(0, A) \big)\mathcal{D}_t^f l(t,s)\dif s,\ \ \ t>0, \ A\in \mathcal{B}(\mathbb{R}).
	\end{equation}
	 %Hereafter we denote by $l_f'$ and $ p_f'$ the derivatives of $l_f$ and $p_f$ with respect to the first variable (representing the time).\\
	 Fix $t>0$. For hypothesis \eqref{ipotesiI} for each $0<a<t<c$ the derivative $l'(r,\cdot)$ is uniformly bounded for $a\le r\le c$ by an integrable function $g_{a,c}$, therefore we can exchange derivative and integral in the following equation,
	 \begin{align}\label{scambiabilitaDerivataIntegraleLeggeInversoSubordinatore}
	 	\int_0^\infty  l'(t,s)\dif s =\frac{\partial }{\partial t} \int_0^\infty l(t,s)\dif s = 0, \ \ \ t>0.
	 \end{align}
	 By \eqref{scambiabilitaDerivataIntegraleLeggeInversoSubordinatore}, and since $|p(t,A)|\le 1, \ t\ge0, A\in\mathcal{B}(\mathbb{R})$, we can write,
	 \begin{align}
	 	\mathcal{D}_t^f p_f(t,A) &= b\frac{\partial }{\partial t}p_f(t,A) + \int_0^t p_f'(t-s,A) \nu(s)\dif s \nonumber\\
	 	& = b\int_0^\infty p(s, A) l'(t,s)\dif s + \int_0^t p_f'(w,A) \nu(t-w)\dif w\nonumber\\
	 	& = b\int_0^\infty\big( p(s, A)-p(0,A)\big) l'(t,s)\dif s + \int_0^t p_f'(w,A) \nu(t-w)\dif w.\label{formulaDaDimostrarePerScambiablitaIntegraliEDerivata}
	 \end{align}
	 Note that the (first) relationship in \eqref{scambiabilitaDerivataIntegraleLeggeInversoSubordinatore} is not sufficient to admit the exchangeability of the derivative and the integral in the second term of \eqref{formulaDaDimostrarePerScambiablitaIntegraliEDerivata} (represented by $p_f'$) because \eqref{ipotesiI} provides a uniform bound only on the compact sets excluding $t=0$ and for this time we have $l(0,s) = \delta_0(s)$, therefore the derivative $l'(t,s)$ has a singularity for $t,s\rightarrow0$ which may prevent convergence (hypothesis \eqref{ipotesiII} helps in solving this issue by providing some control around this singularity).
	 
	 Now, from \eqref{formulaDaDimostrarePerScambiablitaIntegraliEDerivata}, to obtain \eqref{scambiabilitaOperatoreConvoluzionaleEIntegrale} we must still prove that $ \int_0^t p_f'(w,A) \nu(t-w)\dif w = \int_0^\infty p(s,A)\dif s \int_0^t l'(w,s)\nu(t-w)\dif w$.
	 
	 Thanks to \eqref{scambiabilitaDerivataIntegraleLeggeInversoSubordinatore} we can write (remember that by Condition I, $\nu(s)<\infty$ for $s>0$, and $\nu(0)=\infty$),
	 \begin{align}\label{uguaglianzaPrimaDidisuguaglianzaIntegrale}
	 	 p_f'(w,A) \nu(t-w) =  \nu(t-w)\int_0^\infty \big( p(s, A)-p(0,A)\big) l'(w,s)\dif s, \ \ \ t>w>0.
	 \end{align}
	 Therefore, without any hypothesis of integrability, \eqref{uguaglianzaPrimaDidisuguaglianzaIntegrale} yields
	 \begin{equation}\label{disuguaglianzaSufficientePerFubini}
	  \int_0^t \big| p_f'(w,A)\big| \nu(t-w) \dif w \le \int_0^t \nu(t-w)\dif w\int_0^\infty \big| p(s, A)-p(0,A)\big|\, \big|l'(w,s)\big|\dif s, \ \ \ t>0.
	 \end{equation}
	 If the right-hand side of \eqref{disuguaglianzaSufficientePerFubini} is finite we can apply Fubini's theorem and \eqref{uguaglianzaPrimaDidisuguaglianzaIntegrale} would hold also after integration, meaning that we would have	 
	\begin{align}
		\int_0^t p_f'(w,A) \nu(t-w) \dif w &=\int_0^t  \nu(t-w) \dif w \int_0^\infty \big( p(s, A)-p(0,A)\big) l'(w,s)\dif s\nonumber\\
		& = \int_0^\infty \big( p(s, A)-p(0,A)\big)\dif s \int_0^t  l'(w,s) \nu(t-w) \dif w \nonumber\\
		& = \int_0^\infty \big( p(s, A)-p(0,A)\big) \mathcal{D}_t^f l(t,s)\dif s,\nonumber
	\end{align}
	which would prove \eqref{scambiabilitaOperatoreConvoluzionaleEIntegrale}. Thus, we now show that the right-hand side of \eqref{disuguaglianzaSufficientePerFubini} is finite. \\
	First, we observe that since $M\sim CP$ with constant rate $\lambda$, then for each $A\in\mathcal{B}(\mathbb{R})$, $|p(t,A) - p(0,A)|\le 1-e^{-\lambda t}\le \lambda t$ and therefore it exists $C = \max\{1,\lambda\}>0$ such that $|p(t,A) - p(0,A)|\le (1\wedge t)C, \ t\ge0, \ A\in\mathcal{B}(\mathbb{R})$.
	\\Also note that $\nu$, being the tail of the L\'{e}vy measure $\bar\nu$, is non-increasing and by using the property $\int_0^\infty (1\wedge s)\bar\nu(\dif s)<\infty$ it is easy to show that $\int_\varepsilon^t \nu(t-w)\dif w <\infty, \ 0<\varepsilon\le t.$
	\\Finally, observe that the inequality in hypothesis \eqref{ipotesiII} still holds by setting the outer integral over $(0,\varepsilon)$ with $0<\varepsilon\le\delta$ (because the integrand function is non-negative). Thus, for any $t>0$, there exists $0<\varepsilon\le\min\{\delta, t\}$ and we have the following,
	\begin{align}
		\int_0^t& \nu(t-w)\dif w\int_0^\infty \big| p(s, A)-p(0,A)\big|\, \big|l'(w,s)\big|\dif s \nonumber\\
		& \le \int_0^\varepsilon \nu(t-w)\dif w\int_0^\infty (1\wedge s)C \, \big|l'(w,s)\big|\dif s + \int_\varepsilon^t \nu(t-w)\dif w\int_0^\infty (1\wedge s)C \, \big|l'(w,s)\big|\dif s \nonumber\\
		& \le \nu(t-\varepsilon)C \int_0^\varepsilon \dif w\int_0^\infty (1\wedge s) \, \big|l'(w,s)\big|\dif s + C\int_\varepsilon^t \nu(t-w)\dif w\int_0^\infty \, \big|l'(w,s)\big|\dif s \nonumber\\
		& \le \nu(t-\varepsilon)C \int_0^\varepsilon \dif w\int_0^\infty (1\wedge s) \, \big|l'(w,s)\big|\dif s + C\int_\varepsilon^t \nu(t-w)\dif w\int_0^\infty \, g_{\varepsilon, t}(s)\dif s <\infty, \nonumber
	\end{align}
	where we used both \eqref{ipotesiI} and \eqref{ipotesiII}. This completes the proof of \eqref{scambiabilitaOperatoreConvoluzionaleEIntegrale}.	  
	 
	 Now, with \eqref{scambiabilitaOperatoreConvoluzionaleEIntegrale} at hand, we compute the following, for $t>0$ and $A\in \mathcal{B}(\mathbb{R})$,
	\begin{align*}
		\mathcal{D}_t^f p_f(t, A) &= \int_0^\infty \big(p(s,A)- p(0, A) \big) \mathcal{D}_t^f l(t,s)\dif s \nonumber\\
		& =  \int_0^\infty \big(p(s,A)- p(0, A) \big) \Big( \mathbb{D}_t^f l(t,s) -\nu(t) l(0,s) \Big)\dif s \nonumber\\
		& =  \int_0^\infty \big(p(s,A)- p(0, A) \big)  \Big( - \frac{\partial}{\partial s}l(t,s) -\nu(t) l(0,s) \Big) \dif s \nonumber\\
		& = - \int_0^\infty \big(p(s,A)- p(0, A) \big) \frac{\partial}{\partial s}l(t,s) -\nu(t)\int_0^\infty \big(p(s,A)- p(0, A) \big) \delta_0(s) \dif s \nonumber\\
		& =  -\int_0^\infty \big(p(s,A)- p(0, A) \big) \frac{\partial}{\partial s}l(t,s)\dif s \nonumber\\
		& = -\big(p(s,A)- p(0, A) \big) l(t,s)\Big|_{0}^\infty + \int_0^\infty l(t,s) \frac{\partial}{\partial s} p(s,A)  \dif s\\
		&= - \lambda \int_0^\infty (I-C_F)p(s,A)  l(t,s)\dif s
	\end{align*}
	where in the fourth step we used the linearity since both integrals are finite and in the penultimate step we integrated by parts thanks to hypothesis \eqref{ipotesiIII} (in fact $p(\cdot, A)l(t,\cdot )\in AC_{loc}(0,\infty)$ and all the involved integrals converge since \eqref{finitezzaDerivataLeggePoissonComposto}, which is a uniform bound in $t$ because here we deal with a homogeneous compound Poisson, and the second part of \eqref{ipotesiIII}). In the last step we used formula \eqref{equazioneGeneralePoissonCompostoBoreliani} The proof completes by observing that we can exchange the integral and the operator $(I-C_F)$ since $C_F$ is the convolution with respect to a probability distribution $F$. 
\end{proof}

%The proof uses the same arguments presented in the proof of Theoerem 4.1 of \cite{CO2026} and therefore it is omitted (also see Theorem 1 of \cite{BS2024}). 

We point out that Proposition \ref{proposizioneComposizioneInversoSubordinatore} is a special case of Theorem 4.1 of \cite{CO2025}. Indeed, it suffices to observe that $l_f$ satisfies problem (4.4) of \cite{CO2025} (restricted to $x>0$) suitably adapted to the operator \eqref{definizioneDervitaConvoluzionale} and that this satisfies hypothesis (4.1) of \cite{CO2025}.
%$$\int_0^\infty e^{-\mu t} \mathcal{D}_t^f u(t,x)\dif x = .$$

\begin{remark}[Inverse of the stable subordinator]
	We show that the probability density of the inverse of a stable subordinator of order $\alpha\in (0,1)$ satisfies the hypotheses \eqref{ipotesiI}, \eqref{ipotesiII} and \eqref{ipotesiIII}. In this case the probability law reads
	\begin{equation}
		l_\alpha(t,s) = t^{-\alpha} M_\alpha(st^{-\alpha}),\ \ \ t>0, s\ge0,
	\end{equation}
	where $M_\alpha$ is the $M$-Wright function of order $\alpha\in (0,1)$, defined by $M_\alpha(z)= W_{-\alpha, 1-\alpha} (-z) = \sum_{k=0}^\infty (-z)^k/(k!\Gamma(-\alpha k +1-\alpha)),\ z\in\mathbb{C}$, with $W_{\beta,\gamma}$ being the Wright function, $\beta>-1, \gamma\in\mathbb{C}$.
	
	We observe that $M_\alpha$ is an entire function (analytic over $\mathbb{C}$), thus also its derivatives are, with an exponential decay as $s\longrightarrow \infty$, for integer $k\ge0$,
	\begin{equation}\label{decadimentoEsponenzialeFunzioneMEDerivate}
	 \big|M^{(k)}_\alpha (s)\big| \sim C_{\alpha,k}' s^{p} e^{-C_{\alpha,k}''s^{q}}, \ \ \ C_{\alpha,k}',C_{\alpha,k}''>0,
	\end{equation}
	with suitable $p\in\mathbb{R}$ and $q>0$ (for instance if $k=0$, $p = (\alpha - 1/2)(1-\alpha)$ and $q = 1/(1-\alpha)$), see for instance (4.5) of \cite{MMP2010} and references therein for more details.
	
	%The asymptotic behavior \eqref{decadimentoEsponenzialeFunzioneMEDerivate} allows us to prove that $l_\alpha$ satisfies the hypotheses of Proposition \ref{proposizioneComposizioneInversoSubordinatore}.
	
	Hypothesis \eqref{ipotesiII}. For $t>0$ and $s\ge0$, we have
	\begin{align}\label{derivataInTempoInversoSubordinatoreStabile}
		\frac{\partial }{\partial t}l_\alpha(t, s) =   -\alpha t^{-\alpha-1}\left(M_\alpha(st^{-\alpha}) + st^{-\alpha} M'(st^{-\alpha})\right). 
	\end{align}
	Let $\delta >0$, for the first term we have ($M_\alpha$ is a non-negative function)
	\begin{align}
	 	\int_0^\delta \dif t \int_0^\infty (1\wedge s) t^{-\alpha-1} | M_\alpha(st^{-\alpha}) | \dif s &\le \int_0^\delta \int_0^\infty (yt^{\alpha}) t^{-\alpha-1} M_\alpha(y)\dif y t^\alpha \nonumber\\
	 	& = \int_0^\delta t^{\alpha-1} \dif t\int_0^\infty y M_\alpha(y)\dif y <\infty\label{inversoSubordinatoreStabileSoddisfaIpotesiII}
	\end{align}
	where the integral in $t$ is finite since $\alpha-1>-1$ and the integral in $y$ converges because $M_\alpha$ is continuous and bounded in $y = 0$ and for $y\rightarrow\infty$ the exponential decay \eqref{decadimentoEsponenzialeFunzioneMEDerivate} ensures convergence. The second term of \eqref{derivataInTempoInversoSubordinatoreStabile}, involving $M_\alpha'$, is studied in the same (note that $M_\alpha'$ can take negative values, but nothing changes thanks to its regularity and its exponential decay).
	
	Hypothesis \eqref{ipotesiIII}. 
	Set $t>0$ and $0<a<c$. $\partial /(\partial s) l_\alpha(t,s) = t^{-2\alpha}M_\alpha'(st^{-\alpha})$. Thus, by observing that both $M_\alpha(st^{-\alpha})$ and $M_\alpha'(st^{-\alpha})$ are continuous and bounded for $s\in[a,c]$, then $l_\alpha$ is locally Lipschitz and thus locally absolutely continuous. Moreover,
	$$ \int_0^\infty \Big| \frac{\partial }{\partial s}l_\alpha(t, s) \Big|  \dif s= \int_0^\infty t^{-\alpha} \big| M'_\alpha(y) \big| \dif y <\infty $$
	by means of the considerations used for \eqref{inversoSubordinatoreStabileSoddisfaIpotesiII}.

Hypothesis \eqref{ipotesiI}. Let $0<a<c$. From \eqref{derivataInTempoInversoSubordinatoreStabile}, for $s\ge0$ we have
	\begin{align}
		\Big| \frac{\partial }{\partial t}l_\alpha(t, s) \Big|&\le \alpha a^{-\alpha}\Big( M_\alpha(st^{-\alpha}) + sa^{-\alpha} \big|M'(st^{-\alpha})\big|\Big),\ \ \ t\in[a,c].
	\end{align} 
	Given that $M_\alpha$ and $|M_\alpha'|$ are continuous in $[0,\infty)$ with asymptotic behavior \eqref{decadimentoEsponenzialeFunzioneMEDerivate}, it is sufficient to prove the following statement.
	\\
	
	\textbf{Statement I.} Let $\alpha>0$ and $f:[0,\infty)\longrightarrow\mathbb{R}$ be a continuous function such that $\lim_{s\rightarrow\infty} |f(s)|/h(s) = C>0$ exists finite, where $h:[0,\infty)
\longrightarrow[0,\infty)$ is a non-negative function such that $\exists\ M>0$ for which $h$ is non-increasing in $[M,\infty)$ and $\int_M^\infty h(s)\dif s<\infty$. Then, for any $0<a<c$ there exists $ g_{a,c}\in L^1(0,\infty)$ such that $|f(\cdot\, t^{-\alpha})|\le g_{a,c}(\cdot),\ t\in[a,c]$.
\\

We prove Statement I. Fix $\varepsilon = 1$, then there exists $N>0$ such that $\big||f(s)|/h(s) - C\big|<1$ for $s\ge N$, thus, $|f(s)|< h(s)(1+C),\ s\ge N$. Let $L = \max\{M,N\}$, $s_{a,c} = \max_{0\le s\le L(c/a)^\alpha} |f(s)|<\infty $ (since $f$ continuous) and define $g_{a,c}(s) = s_{a,c}$ if $s\le Lc^\alpha$ and $g_{a,c}(s) = h(sc^{-\alpha})(1+C)$ if $s> Lc^\alpha$. We show that $g_{a,c}\in L^1(0,\infty)$ and it uniformly dominates $|f(st^{-\alpha})|$ for all $t\in[a,c]$. 
$$ \int_0^t g_{a,c}(s)\dif s = \int_0^{Lc^{\alpha}} s_{a,c} \dif s + \int_{Lc^{\alpha}}^\infty h(sc^{-\alpha})(1+C)\dif s = s_{a,c} Lc^\alpha + (1+C)c^\alpha\int_{L}^\infty h(y)\dif y  <\infty$$
since the hypothesis on the integral of $h$ in $(M,\infty)$ (and $M\le L$). Finally, if $0\le s\le Lc^\alpha$,
$$ |f(st^{-\alpha})| \le \max_{0\le y \le Lc^{\alpha}} |f(y t^{-\alpha}) | = \max_{0\le w \le L(c/t)^{\alpha}} |f(w) | \le \max_{0\le w \le L(c/a)^{\alpha}} |f(w) | = s_{a,c} = g(s), \ \ t\in[a,c],$$
and if $s> Lc^{\alpha}$, then $st^{-\alpha}\ge sc^{-\alpha} > L \ge N$, thus $ |f(st^{-\alpha})| < h(st^{-\alpha})(1+C) \le h(sc^{-\alpha})(1+C) = g(s) $ (since $h$ non-increasing after $M\le L$)  uniformly for $t\in[a,c]$.
\hfill $\diamond$
\end{remark}

\section{Integral of compound processes}

Let $T>0$ and measurable $f:[0,\infty)\longrightarrow\mathbb{R}$. In the present section we consider the following assumption (unless otherwise stated).
\\

\textbf{Condition II.} $f\in L_{loc}^1[0,\infty)$ (i.e. $f\in L^1[0,T]$). $N$ is a counting process with arrival times $S_k,\ k\ge1$ ($S_0 = 0$ a.s) and which cannot explode in finite time intervals ($\sum_{k=1}^\infty (S_k-S_{k-1}) = \infty$ a.s.).
\\

Let $Y$ be a discrete real stochastic process. We are interested in integrals over $[0,t],\ 0\le t\le T,$ which, under Condition II, can be written as (see Appendix \ref{appendiceCalcoloEsplicitoIntegrali} for details)
\begin{equation}\label{formulazioneGeneraleIntegraleProcessoTempoContinuo}
\int_0^t Y_{N(s)} f(s)\dif s = \sum_{n=0}^{N(t)-1} Y_n\int^{S_{n+1}}_{S_n} f(s)\dif s +Y_{N(t)}\int_{S_{N(t)}}^t f(s)\dif s.
\end{equation}

If $Y$ is a random walk, i.e. $Y_n = \sum_{k=1}^n X_k,\ n\in\mathbb{N}_0,$ with $X_k$ i.i.d. and $N$ is independent of $X_k\ \forall\ k$, then $Y_{N}$ is a compound process and the integral \eqref{formulazioneGeneraleIntegraleProcessoTempoContinuo} turns into
\begin{equation}\label{formulazioneGeneraleIntegraleProcessoCompound}
 \int_0^t \sum_{k=1}^{N(s)} X_k f(s) \dif s = \sum_{k=1}^{N(t)} X_k \int^t_{S_k} f(s) \dif s,\ \ \ 0\le t\le T,
\end{equation}
see Appendix \ref{appendiceCalcoloEsplicitoIntegrali} for some details.

Note that, the characteristic function of the integral \eqref{formulazioneGeneraleIntegraleProcessoCompound} can be written as follows. By taking into account the independence of the $X$s, we have that, with $\gamma\in\mathbb{R}$, %(with $\mathbb{E}_Y$ denoting the expectation with respect to the $Y$s only),
\begin{align}
	\mathbb{E} \exp\Bigg(i\gamma \int_0^t\sum_{k=1}^{N(s)} X_k f(s) \dif s\Bigg) &= \sum_{n\ge0} P\{N(t) = n\} \mathbb{E}\Bigg[\prod_{k=1}^n\mathbb{E} \Big[ e^{i\gamma X_k \int_{S_k}^tf(s)\dif s} \,\Big|\,S_1,\dots,S_n\Big]\,|\,N(t) = n\Bigg]\nonumber\\
	& =  \sum_{n\ge0} P\{N(t) = n\} \mathbb{E}\Bigg[\prod_{k=1}^n H_X\bigg(\gamma \int_{S_k}^tf(s)\dif s\bigg) \,|\,N(t) = n\Bigg]. \label{funzioneCaratteristicaIntegraleGenerale}
\end{align}

If $f\equiv 1$, we recover the following result on the increments of the integral.

\begin{lemma}\label{lemmaIncrementiIntegraleProcesso}
	Let $Y$ be a stochastic process on $\mathcal{D}[0,T]$ such that $Y(0) = 0$ a.s.% and integrable over $[0,T]$
	. Let $I(t) = \int_0^t Y(s)\dif s,\ 0\le t\le T$, and $0\le t_1<t_2\le T, \ \gamma \in\mathbb{R}$.
	\begin{itemize}
		\item[($i$)] If $Y$ has independent increments,
		\begin{equation}\label{incrementoIntegraleProcessoStocasticoIncrementiIndipendenti}
			\mathbb{E}e^{i\gamma \big(I(t_2)-I(t_1)\big)} = \mathbb{E}e^{i\gamma (t_2-t_1)Y(t_1)}\mathbb{E}e^{i\gamma \int_0^{t_2-t_1} \big(Y(s+t_1)-Y(t_1)\big)\dif s}.
		\end{equation}
		\item[($ii$)] If $Y$ has independent and stationary increments,
		\begin{equation}\label{incrementoIntegraleProcessoStocasticoIncrementiIndipendentiStazionario}
			\mathbb{E}e^{i\gamma \big(I(t_2)-I(t_1)\big)} = \mathbb{E}e^{i\gamma (t_2-t_1)Y(t_1)}\mathbb{E}e^{i\gamma I(t_2-t_1)}.
		\end{equation}
	\end{itemize}
\end{lemma}
See Appendix \ref{appendiceDimostrazioneIncrementiIntegrale} for the proof of Lemma \ref{lemmaIncrementiIntegraleProcesso}.

\subsection{Poisson case}

For the sake of completeness we recall the following result which follows by classical arguments.

\begin{lemma}\label{lemmaLimiteSerieRappresentanteIntegrale}
	Let $g:\mathbb{N}^2\longrightarrow \mathbb{R}$ and $k:\mathbb{R}\longrightarrow\mathbb{R}$ such that for $\varepsilon >0\ \exists\ N_\varepsilon$ s.t. $|g(h,n) - k(h/n)|<\varepsilon,\ n\ge N_\varepsilon,\ h=1,\dots, n$. Let $f:\mathbb{R}\longrightarrow\mathbb{R}$ s.t. $\int_0^1 |f(x)|\dif x<\infty$ and $\exists\ \int_0^1 f(x)k(x)\dif x$. Then, 
	\begin{equation*}
		\lim_{n\rightarrow\infty} \sum_{h=1}^n \frac{1}{n}f\Big(\frac{h}{n}\Big) g(h,n) = \int_0^1f(x)k(x)\dif x.
	\end{equation*}
\end{lemma}

Now we prove that the fractional integral of a compound Poisson process (up to a constant) is a compound Poisson random variable. Note that this result concerns the integral over a set $[0,t]$ for fixed $t>0$ and in Remark \ref{fallimentoUguaglianzaSpazionTraiettorie} is shown that this representation does not hold on $\mathcal{D}[0,T]$.

\begin{theorem}\label{teoremaIntegraleProcessoCompostoPoisson}
Under Condition II, let $M\sim CP(\lambda, X)$ and $\alpha\ge0$. For $t>0$, set $\Lambda_t$ a r.v. such that
\begin{equation}\label{distribuzioneVariabileAleatoriaLambda}
	P\{\Lambda_t\in \dif s\} = \frac{\lambda(s)\dif s}{\int_0^t \lambda(u)\dif u}, \ \ \ s\in[0,t].
\end{equation}
Then,
\begin{equation}\label{integraleProcessoCompostoPoisson}
	\int_0^t (t-s)^\alpha M(s)\dif s \stackrel{d}{=} \sum_{k=1}^{N(t)} X_k\frac{(t-\Lambda_{t,k})^{\alpha+1}}{\alpha+1},\ \ \ t\ge0,
\end{equation}	
where $\Lambda_{t,1}, \Lambda_{t,2},\dots$ are independent copies of $\Lambda_t$, independent of $N$ and of $X_1,X_2,\dots$.
\end{theorem}

\begin{proof}
	First, we observe that the characteristic function of the right hand-side of \eqref{integraleProcessoCompostoPoisson} is, for $\gamma \in\mathbb{R}$,
	\begin{equation}
		\mathbb{E}\exp\Bigg( i\gamma \sum_{k=1}^{N(t)} X_k\frac{(t-\Lambda_{t,k})^{\alpha+1}}{\alpha+1} \Bigg)= e^{-\int_0^t \lambda(s)\dif s} \exp\Bigg( \int_0^t \lambda(s) \mathbb{E} e^{i\gamma X \frac{(t-s)^{\alpha+1}}{ \alpha+1} }\dif s \Bigg).
	\end{equation}
	Now, we study the characteristic function of the left hand-side of \eqref{integraleProcessoCompostoPoisson}.
\begin{align}
\int_0^t (t-s)^\alpha M(s)\dif s & = 	\lim_{n\rightarrow \infty} \sum_{k=1}^n \bigg(t-\frac{kt}{n}\bigg)^\alpha M\bigg(\frac{kt}{n}\bigg)\frac{t}{n}\nonumber \\
& =  \lim_{n\rightarrow \infty} \sum_{h=1}^n \Bigg(M\bigg(\frac{ht}{n}\bigg) - M\bigg(\frac{(h-1)t}{n}\bigg) \Bigg) \sum_{k = h}^n \bigg(\frac{t}{n}\bigg)^{\alpha+1}(n-k)^\alpha \nonumber \\
&= \lim_{n\rightarrow \infty} \sum_{h=1}^n \Bigg(M\bigg(\frac{ht}{n}\bigg) - M\bigg(\frac{(h-1)t}{n}\bigg) \Bigg) \sum_{j = 0}^{n-h} \bigg(\frac{t}{n}\bigg)^{\alpha+1} j^\alpha .\label{primaRappresentazioneIntegraleFrazinarioCompound}
\end{align}

Note that for $n$ sufficiently large, then, for $ h=1,\dots,n$ and $\gamma \in\mathbb{R}$,
\begin{align}\label{funzioneCaratteristicaIncrementoCompoundPoisson}
	\mathbb{E} \exp\Bigg(i\gamma \bigg[ M\bigg(\frac{ht}{n}\bigg) - M\bigg(\frac{(h-1)t}{n}\bigg)\bigg] \Bigg) =1-\lambda\bigg(\frac{(h-1)t}{n}\bigg)\frac{t}{n}\Bigg(1-\mathbb{E}e^{i\gamma X}\Bigg)+o\bigg(\frac{t}{n}\bigg).
\end{align}
Then, the characteristic function of \eqref{primaRappresentazioneIntegraleFrazinarioCompound} reads
\begin{align}
	\mathbb{E}\exp\bigg(&i\gamma \int_0^t (t-s)^\alpha M(s)\dif s\bigg)\nonumber \\
	&= \lim_{n\rightarrow\infty}\prod_{h=1}^n\mathbb{E} \exp\Bigg(i\gamma \bigg[M\bigg(\frac{ht}{n}\bigg) - M\bigg(\frac{(h-1)t}{n}\bigg) \bigg] \sum_{j = 0}^{n-h} \bigg(\frac{t}{n}\bigg)^{\alpha+1} j^\alpha \Bigg)\nonumber\\
	& = \lim_{n\rightarrow\infty}\prod_{h=1}^n \Bigg( 1-\lambda\bigg(\frac{(h-1)t}{n}\bigg)\frac{t}{n}\bigg[1-\mathbb{E}\exp\bigg( i\gamma X \sum_{j = 0}^{n-h} \bigg(\frac{t}{n}\bigg)^{\alpha+1} j^\alpha \bigg)\bigg] \Bigg)\nonumber\\
	& = \exp\Bigg( -\lim_{n\rightarrow\infty} \sum_{h=1}^n\lambda\bigg(\frac{(h-1)t}{n}\bigg)\frac{t}{n}\bigg[1-\mathbb{E}\exp\bigg( i\gamma X \sum_{j = 0}^{n-h} \bigg(\frac{t}{n}\bigg)^{\alpha+1} j^\alpha \bigg)\bigg] \Bigg) \nonumber\\ % COMMENTA
	& = e^{-\int_0^t \lambda(s)\dif s} \exp\Bigg(\lim_{n\rightarrow\infty} \sum_{h=1}^n \lambda\bigg(\frac{(h-1)t}{n}\bigg)\frac{t}{n} \mathbb{E}\exp\bigg( i\gamma X \sum_{j = 0}^{n-h} \bigg(\frac{t}{n}\bigg)^{\alpha+1} j^\alpha \bigg) \Bigg) ,\nonumber
\end{align}
where in the first step we used \eqref{primaRappresentazioneIntegraleFrazinarioCompound}, the dominated convergence theorem to exchange limit and expectation and the independence of the increments of $M$.  In the second step we used \eqref{funzioneCaratteristicaIncrementoCompoundPoisson}. In the third step we used that if $\lim_{n\rightarrow\infty} \sum_{k=1}^n \frac{a_{k,n}}{n}<\infty$ and $\sum_{k=1}^n \frac{a^2_{k,n}}{n^2}=0$ (which can be easily checked), then
$$\lim_{n\rightarrow\infty} \prod_{k=1}^n\Big(1+\frac{a_{k,n}}{n}\Big) = \exp\Bigg( \lim_{n\rightarrow\infty}\sum_{k=1}^n \frac{a_{k,n}}{n}\Bigg).$$

Now, it remains to prove that
\begin{equation}
	\lim_{n\rightarrow\infty} \sum_{h=1}^n \lambda\bigg(\frac{(h-1)t}{n}\bigg)\frac{t}{n} \mathbb{E}\exp\bigg( i\gamma X \sum_{j = 0}^{n-h} \bigg(\frac{t}{n}\bigg)^{\alpha+1} j^\alpha \bigg) = \int_0^t \lambda(s) \mathbb{E} e^{i\gamma X \frac{(t-s)^\alpha}{ \alpha+1}  } \dif s
	\end{equation}

By Lemma \ref{lemmaLimiteSerieRappresentanteIntegrale}, it suffices to show that $ \sum_{j = 0}^{n-h} \big(\frac{t}{n}\big)^{\alpha+1} j^\alpha \sim \frac{(t-ht/n)^{\alpha+1}}{\alpha+1}$ as $n$ tends to infinity. Indeed,

\begin{align}
	\sum_{j=0}^{n-h} \frac{j^\alpha}{n^{\alpha+1}}  & = \frac{1}{n}\sum_{j=0}^{n-h}\sum_{k\ge0} \frac{\Gamma(\alpha+1)}{k!\Gamma(\alpha+1-k)}  \Big(\frac{j}{n}-1\Big)^k \nonumber \\
	&= \frac{1}{n}\sum_{j=0}^{n-h}\sum_{k\ge0} \frac{\Gamma(\alpha+1)}{k!\Gamma(\alpha+1-k)}  \sum_{i=0}^k \binom{k}{i} (-1)^{k-i} \Big(\frac{j}{n}\Big)^i \nonumber\\ % COMMENTA
	& = \frac{1}{n}\sum_{k\ge0} \frac{\Gamma(\alpha+1)}{k!\Gamma(\alpha+1-k)}  \sum_{i=0}^k \binom{k}{i} (-1)^{k-i} \sum_{j=0}^{n-h}\Big(\frac{j}{n}\Big)^i \nonumber\\
	&= \sum_{k\ge0} \frac{\Gamma(\alpha+1)}{k!\Gamma(\alpha+1-k)}  \sum_{i=0}^k \binom{k}{i} (-1)^{k-i} \frac{1}{i+1}\sum_{j=0}^i \binom{i+1}{j} (-1)^j B_j \frac{(n-h)^{i+1-j}}{n^{i+1}}, \label{espansioneSommatoriaPotenzaReale} %\label{passaggioFormulaFaulhaber}
\end{align}
where in the last step we used the Faulhaber's formula $\sum_{j=1}^m j^i = \sum_{j=0}^i \binom{i+1}{j} (-1)^j B_j m^{i+1-j} / (i+1)$, $i\ge1$, where $B_j$ are the Bernoulli numbers with $B_0 = 1$. Then, for $n$ big enough the last sum in \eqref{espansioneSommatoriaPotenzaReale} reduces to the term with $j=0$, i.e.
\begin{align}
	\sum_{j=0}^{n-h} \frac{j^\alpha}{n^{\alpha+1}}  & \sim \sum_{k\ge0} \frac{\Gamma(\alpha+1)}{k!\Gamma(\alpha+1-k)}  \sum_{i=0}^k \binom{k}{i} (-1)^{k-i} \frac{1}{i+1} \bigg(1-\frac{h}{n}\bigg)^{i+1}  \nonumber\\
	& = \frac{1}{\alpha+1}\sum_{k\ge0} \frac{\Gamma(\alpha+2)}{(k+1)!\Gamma(\alpha+1-k)} (-1)^{k+1} \sum_{i=0}^k \binom{k+1}{i+1} (-1)^{i+1} \bigg(1-\frac{h}{n}\bigg)^{i+1}  \nonumber\\
	& =  \frac{1}{\alpha+1}\sum_{k\ge0} \frac{\Gamma(\alpha+2)}{(k+1)!\Gamma(\alpha+1-k)} (-1)^{k+1} \Bigg( \bigg(\frac{h}{n}\bigg)^{k+1}- 1\Bigg) \nonumber\\
	& = \frac{1}{\alpha+1} \bigg(1-\frac{h}{n}\bigg)^{\alpha +1},\nonumber
\end{align}
which completes the proof.
\end{proof}

\begin{remark}[Iterated integrals]
	In light of the Cauchy integral formula, the iterated integral of $M\sim CP(\lambda, X)$ is, with $t\ge0$ and $n\in\mathbb{N}$,
	\begin{align*}
		\int_0^t\dif s_1\int_0^{s_1} \dif s_2\cdots \int_0^{s_{n-1}}M(s_n)\dif s_n &= \frac{1}{(n-1)!}\int_0^t M(s)(t-s)^{n-1} \dif s \nonumber\\
		& \stackrel{d}{=} \frac{t^n}{n!} \sum_{k=1}^{N(t)} X_k \Big(1-\frac{\Lambda_{t,k}}{t}\Big)^{n} \stackrel{p}{\longrightarrow}0, \ \ \ n\rightarrow\infty,
	\end{align*}
	where the limit follows since $N$ does not explode (see Condition II) and $0<1-\Lambda_{t,k}/t<1$ a.s. \hfill $\diamond$
\end{remark}

\begin{remark}[Homogeneous case]\label{remarkIntegraleCompoundOmogeneo}
	If $\lambda(t) = \lambda ,\ t\ge0$, the formula \eqref{integraleProcessoCompostoPoisson} simplifies. Indeed, $\Lambda_t \sim Unif(0,t)$ and $t-\Lambda_t \sim Unif(0,t)$. Thus, denoting by $U_k\sim Unif(0,1)$ i.i.d. r.v.s independent of $N$ and $X_1,\dots$,
	\begin{equation}\label{integraleFrazionarioCompoundOmogeneo}
		\int_0^t M(s)(t-s)^\alpha \dif s \stackrel{d}{=} \frac{t^{\alpha+1}}{\alpha+1} \sum_{k=1}^{N(t)} X_k U_k^{\alpha+1},\ \ \ t\ge0,\ \alpha\ge0.
	\end{equation}
	Note that the case of $\alpha = 0$ already appeared in \cite{CO202}.
	
	We readily derive the following limit result. Without any loss of generality, let $\mathbb{E}X =0$ and $\mathbb{E}X^2 <\infty$, then as $t\longrightarrow \infty$,
	\begin{align*}
		\lim_{t\rightarrow\infty} \frac{1}{t^{\beta}} \int_0^t M(s) \dif s \stackrel{d}{=} \frac{1}{t^{\beta-1}} \sum_{k=1}^{N(t)} X_k U_k 	
		 \longrightarrow \begin{cases}
			\begin{array}{l l}
				\displaystyle \stackrel{p}{\longrightarrow} 0, & \beta = 2,\\
				\displaystyle \stackrel{d}{\longrightarrow} Z\sim Normal(0, \lambda \mathbb{E}X^2 / 3\big), & \beta = 3/2,
			\end{array}
		\end{cases}
	\end{align*}
	where we used \eqref{integraleFrazionarioCompoundOmogeneo} and the limit theorems for the sum of a random number of i.i.d. terms. \hfill$\diamond$
\end{remark}

Now, we point out an interesting connection between Theorem \ref{teoremaIntegraleProcessoCompostoPoisson} and formula \eqref{formulazioneGeneraleIntegraleProcessoCompound} with $f(s) = (t-s)^\alpha$. Indeed, the latter turns into
\begin{equation}\label{rappresentazioneQuasiCertantegraleCompoundPoisson}
	\int_0^t M(s)(t-s)^{\alpha}\dif s = \sum_{k=1}^{N(t)} X_k \frac{(t-S_k)^\alpha}{\alpha+1},\ \ \ t\ge0,
\end{equation} 
which resembles \eqref{integraleProcessoCompostoPoisson}, where the interarrival times $S_k$ are replaced by the r.v.s $\Lambda_{t,k}$. The following results explore this idea and permit us to generalize Theorem \ref{teoremaIntegraleProcessoCompostoPoisson}. Formula \eqref{rappresentazioneQuasiCertantegraleCompoundPoisson} is also a particular shot noise process, whose basic formulation is $S(t) = \sum_{k=1}^{N(t)} X_k f(t-S_k)$ with $f$ usually decreasing (see for instance \cite{BCS2026} and references therein).

\begin{proposition}\label{proposizioneDistribuzioneCongiuntaIntertempiPoisson}
	Let $N\sim PP(\lambda)$. Let $t>0$, $\Lambda_{t,1},\dots$ i.i.d r.v.s distributed as in \eqref{distribuzioneVariabileAleatoriaLambda}. Then, for $n\in\mathbb{N}$ and $0\le s_1<s_2<\dots<s_n\le t$,
	\begin{equation}
		P\{S_1\in\dif s_1,\dots,S_n\in \dif s_n\,|\,N(t) = n\} = P\{\Lambda_{t,(1)}\in\dif s_1, \dots, \Lambda_{t, (n)}\in\dif s_n\},
	\end{equation}
	with $\Lambda_{t,(1)},\dots,\Lambda_{t,(n)}$ the (increasingly) ordered $\Lambda_{t,1},\dots,\Lambda_{t,n}$.
\end{proposition}

\begin{proof}
	\begin{align*}
		P\{S_1\in \dif s_1,&\dots,S_n\in \dif s_n\,|\,N(t) = n\} \\
		&= \frac{P\{N[o,s_1) = 0,N[s_1,s_1+\dif s_1) = 1,\dots, N[s_n,s_n+\dif s_n) = 1,N[s_n+\dif s_n,t) = 0\}}{P\{N(t) = n\}} \\
		&= \frac{n!}{\big(\int_0^t\lambda(s)\dif s\big)^n} \prod_{k=1}^n\lambda(s_k)\dif s_k \\
		&= 	P\{\Lambda_{t,(1)}\in\dif s_1, \dots, \Lambda_{t, (n)}\in\dif s_n\}.
	\end{align*}
\end{proof}

\begin{proposition}\label{proposizioneVariabiliOrdinateFunzioniCommutative}
	Let $n\in\mathbb{N}$, $t>0$ and measurable $f:\mathbb{R}\times [0,t] \longrightarrow \mathbb{R}$. Let $\Lambda_1,\dots,\Lambda_n$ i.i.d. r.v.s taking values in $[0,t]$ a.s. and $X_1,\dots,X_n$ independent copies of a real r.v. $X$. Then,
	\begin{equation}
		\sum_{k=1}^n f\Big(X_k, \Lambda_{(k)}\Big) \stackrel{d}{=} \sum_{k=1}^n f\big( X_k,\Lambda_k \big).
	\end{equation}
\end{proposition}

\begin{proof}
	
	Let $\gamma\in\mathbb{R}$, and $\Lambda$ distributed as $\Lambda_k$,
	\begin{align}
		\mathbb{E} \exp\Bigg(i\gamma\sum_{k=1}^n f(X_k,\Lambda_k) \Bigg) &= \left( \mathbb{E} e^{i\gamma f(X,\Lambda) }\right)^n \nonumber\\
		& =  \Bigg( \int_0^t \mathbb{E} e^{i\gamma f(X,\lambda)}P\{\Lambda \in\dif \lambda\} \Bigg)^n \nonumber\\
		& = n! \int_0^t \int_{\lambda_1}^t \cdots \int_{\lambda_{n-1}}^t \prod_{k=1}^n \Bigg(\mathbb{E} e^{i\gamma f(X,\lambda_k)}P\{\Lambda \in\dif \lambda_k\} \Bigg), \label{passaggioUsoFormulaPotenzaIntegrale} \\
		&=  n! \int_0^t \int_{\lambda_1}^t \cdots \int_{\lambda_{n-1}}^t \mathbb{E} \exp\Big(i\gamma \sum_{k=1}^n f(X_k,\lambda_k)  \Big) \prod_{k=1}^n  P\{\Lambda \in\dif \lambda_k\} \nonumber \\
		& = \mathbb{E} \exp\Bigg(i\gamma\sum_{k=1}^n f(X_k,\Lambda_{(k)}) \Bigg),\nonumber
	\end{align}
	%\begin{lemma}	
	where in the second-to-last equality we used that $X_1,\dots,X_n$ are independent copies of $X$ and in \eqref{passaggioUsoFormulaPotenzaIntegrale} we used that, if $f:[0,\infty)\longrightarrow \mathbb{C}$, then, for $t>0$ and $n\in\mathbb{N}$,
	\begin{equation}
		\Biggl(\int_0^t f(x)\dif x \Bigg)^n = n! \int_0^t f(x_1) \dif x_1\int_{x_1}^t f(x_2)\dif x_2 \cdots\int_{x_{n-1}}^t f(x_n)\dif x_n,
	\end{equation}
	which can be proved by induction.
	%\end{lemma}
	\end{proof}
	
Note that the underlying reason for Proposition \ref{proposizioneVariabiliOrdinateFunzioniCommutative} is that the sum is commutative and, thanks to the independence and identical distribution of the $X_k$, the order of the terms does not matter. In light of this observation, the following result generalizes Proposition \ref{proposizioneVariabiliOrdinateFunzioniCommutative}.

\begin{proposition}\label{proposizioneVariabiliOrdinateScambiabilita}
Let $n\in\mathbb{N}$, $t>0$ and the measurable functions $f:\mathbb{R}\times [0,t] \longrightarrow \mathbb{R}$ and $g:\mathbb{R}^n\longrightarrow \mathbb{R}$. Let $\Lambda_1,\dots,\Lambda_n$ be exchangeable r.v.s taking values in $[0,t]$ a.s. and $X_1,\dots,X_n$ be exchangeable real r.v.s independent of $\Lambda$s. If $g$ is commutative (with respect to the inputs), then
\begin{equation}	\label{formulaIndipendentiOrdinatePerCommutazione}
	g\Bigg(f\Big(X_1, \Lambda_{(1)}\Big),\dots, f\Big(X_n, \Lambda_{(n)}\Big) \Bigg) \stackrel{d}{=} g\Bigg(f\Big(X_1, \Lambda_{1}\Big),\dots, f\Big(X_n, \Lambda_{n}\Big) \Bigg).
\end{equation}
\end{proposition}

The proof is given in Appendix \ref{appendiceDimostrazioneProposizioneVariabiliOrdinateScambiabilita}.
\\

Now we can generalize Theorem \ref{teoremaIntegraleProcessoCompostoPoisson}, proving that for fixed $t\in[0, T]$ the integral \eqref{formulazioneGeneraleIntegraleProcessoCompound} of a compound Poisson process is a compound Poisson random variable. Indeed, equation \eqref{formulazioneGeneraleIntegraleProcessoCompound} and the order-statistics argument developed in Propositions \ref{proposizioneDistribuzioneCongiuntaIntertempiPoisson} and \ref{proposizioneVariabiliOrdinateFunzioniCommutative} easily yield the next statement.

\begin{theorem}\label{teoremaIntegraleProcessoCompostoPoissonGenerale}
	Let $T>0$ and measurable $f\in L^1[0,T]$. Under the hypothesis of Theorem \ref{teoremaIntegraleProcessoCompostoPoisson}, 
	\begin{equation}\label{integraleGeneraleProcessoCompostoPoisson}
		\int_0^t  M(s)f(s)\dif s \stackrel{d}{=} \sum_{k=1}^{N(t)} X_k\int_{\Lambda_{t,k}}^t f(s)\dif s,\ \ \ 0\le t \le T.
	\end{equation}	
\end{theorem}

%\begin{proof}
%	The statement easily follows by means of the showed results. For $\gamma\in\mathbb{R}$,
%	\begin{align*}
%		\mathbb{E}\exp\left(i\gamma \int_0^t M(s)f(s)\dif s \right) &= \mathbb{E}\left[\mathbb{E}\bigg[\exp\bigg(\sum_{k=1}^{N(t)} X_k\int_{S_k}^t f(s)\dif s \bigg)\bigg| N(t)\bigg]\right]\\
%		&=\mathbb{E}\left[\mathbb{E}\bigg[\exp\bigg(\sum_{k=1}^{N(t)} X_k\int_{\Lambda_{t,k}}^t f(s)\dif s \bigg)\bigg| N(t)\bigg]\right],
%	\end{align*}
%	where in the first step we used equation \eqref{formulazioneGeneraleIntegraleProcessoCompound} and in the second one the order-statistics argument developed in Propositions \ref{proposizioneDistribuzioneCongiuntaIntertempiPoisson} and \ref{proposizioneVariabiliOrdinateFunzioniCommutative}.
%\end{proof}

Note that, in view of point ($i$) of Remark \ref{remarkUguaglianzaSuSpazioFunzioniContinueCompoundDiCompound} the compound of a compound Poisson process still is a compound Poisson process (since the equality in distribution over $\mathcal{D}[0,T]$) and therefore Theorem \ref{teoremaIntegraleProcessoCompostoPoissonGenerale} suitably holds for the integral of ``iterated" compound Poisson processes.

Proposition \ref{proposizioneVariabiliOrdinateFunzioniCommutative} also allows to Theorem \eqref{teoremaIntegraleProcessoCompostoPoissonGenerale} to be extended to compound Poisson processes with exchangeable jumps (rather than independent and identically distributed).

Furthermore, similarly to Theorem \eqref{teoremaIntegraleProcessoCompostoPoissonGenerale}, the presented order-statistics argument holds for the shot noise processes as well, meaning that $S(t) = \sum_{k=1}^{N(t)} X_k f(t-S_k) \stackrel{d}{=} \sum_{k=1}^{N(t)} X_k f(t-\Lambda_{t,k})$. 

\begin{remark}\label{fallimentoUguaglianzaSpazionTraiettorie}
	We show that, in general, both relationship \eqref{integraleProcessoCompostoPoisson} and \eqref{integraleGeneraleProcessoCompostoPoisson} do not hold on $\mathcal{D}[0,T]$. It suffices to consider the case of a homogeneous process $M\sim CP(\lambda, X)$ and $f\equiv 1$. By Remark \ref{remarkIntegraleCompoundOmogeneo} and  \eqref{incrementoIntegraleProcessoStocasticoIncrementiIndipendentiStazionario}, with $0\le t_1<t_2\le T$ and $\gamma \in\mathbb{R}$, the increment of the integral is such that
	\begin{equation}\label{incrementoVeroIntegralePoissonCompostoOmogeneo}
		\mathbb{E}e^{i\gamma \int_{t_1}^{t_2} M(s)\dif s} = \exp\Big(-\lambda t_2 + \lambda t_1 \mathbb{E} e^{i\gamma (t_2-t_1)X} +\lambda(t_2 -t_1)\mathbb{E}e^{i\gamma (t_2-t_1) XU}  \Big),
	\end{equation}
	with $U\sim Unif(0,1)$ being independent of $X$. On the other hand, denoting by $Z(t) = t\sum_{k=1}^{N(t)} X_k U_k, \ 0\le t\le T,$ as in \eqref{integraleFrazionarioCompoundOmogeneo}, we have 
	\begin{align}
		\mathbb{E}e^{i\gamma \big(Z(t_2)-Z(t_1)\big) } & = \mathbb{E} \exp\Bigg(i\gamma (t_2-t_1)\sum_{k=1}^{N(t_1)}X_kU_k + i\gamma t_2 \sum_{k=1}^{N(t_2)-N(t_1)} X_kU_k  \Bigg) \nonumber\\
		& = \exp\Bigg( -\lambda t_2 +\lambda t_1\Big( \mathbb{E}e^{i\gamma (t_2-t_1)XU}\big)+\lambda(t_2 -t_1)\mathbb{E}e^{i\gamma t_2 XU} \Bigg), \nonumber %\label{incrementoRappresentazioneIntegralePoissonCompostoOmogeneo}
	\end{align}
	which in general differs from \eqref{incrementoVeroIntegralePoissonCompostoOmogeneo}. \hfill$\diamond$
\end{remark}

\begin{remark}[Limit behaviour]
	Inspired by the Kac's hydrodynamic limit conditions, Theorem \ref{teoremaIntegraleProcessoCompostoPoissonGenerale} permits us to derive the limit result when the rate function of the underlying Poisson process goes to infinity. Let $\beta >0$ and $M\sim CP(\lambda, X)$ with integrable $\lambda$ such that $\lambda(t) = \beta \xi(t),\ t\ge0$. Now, the r.v. $\Lambda_{t}$, having density \eqref{distribuzioneVariabileAleatoriaLambda}, does not depend on $\beta$ (since its distribution does not). Furthermore, $N(t)/\beta \stackrel{P,L^1}{\longrightarrow} \int_0^t\xi(s)\dif s\ \forall\ t$ (see for instance Corollary 2.1 of \cite{CO202}). Hence, we have the following limit result, for $\mathbb{E}X = 0$ and $\mathbb{E}X^2<\infty$,
	\begin{equation}\label{comportamentoLimiteTipoKAc}
		\frac{1}{\beta^\alpha} \int_0^t M(s)f(s)\dif s 
	 \longrightarrow \begin{cases}
			\begin{array}{l l}
				\displaystyle \stackrel{P}{\longrightarrow} 0, & \alpha = 1,\\
				\displaystyle \stackrel{d}{\longrightarrow} Z\sim Normal\Bigg(0, \mathbb{E}X^2 \int_0^t\xi(s)\dif \mathbb{E}\bigg(\int_{\Lambda_t}^t f(s)\dif s\bigg)^2\Bigg), & \alpha = 1/2,
			\end{array}
		\end{cases}
	\end{equation}
	as $\beta\longrightarrow\infty$. Note that using the distribution \eqref{distribuzioneVariabileAleatoriaLambda}, the variance of the Normal limit turns into $\mathbb{E}X^2\int_0^t \xi(u) \dif u\left(\int_u^t f(s)\dif s\right)^2$.
	\\
	
	As an example, inspired by the paper \cite{MS2025} (also see references therein), we consider $\lambda(t) = \beta /t^{\alpha}$ with $\alpha \in(0,1)$ and $f(s) = s^r$ with $r>0$. Thus, $\mathbb{E}\Lambda_t^u = t^u \frac{1-\alpha}{1-\alpha+u}, \ u\ge 1$. In this case, 
	\begin{align*}
		\mathbb{E} \bigg(\int_{\Lambda_t}^t s^r\dif s\bigg)^2 & = \frac{t^{2(r+1)}-2t^{r+1}\mathbb{E}\Lambda_t^{r+1}+\mathbb{E}\Lambda_t^{2(r+1)}}{(r+1)^2} \\
		&=\frac{t^{2(r+1)}}{(r+1)^2} \bigg( 1-\frac{2(1-\alpha)}{1-\alpha+r+1}+\frac{1-\alpha}{1-\alpha+2(r+1)}\bigg).
		\end{align*}
		Finally, from \eqref{comportamentoLimiteTipoKAc},
$$\frac{1}{\sqrt{\beta}} \int_0^t M(s)s^r\dif s \stackrel{d}{\longrightarrow} Z\sim Normal\Bigg(0, \mathbb{E}X^2 \frac{t^{3-\alpha +2r}}{(1-\alpha)(r+1)^2} \bigg( 1-\frac{2(1-\alpha)}{2-\alpha+r}+\frac{1-\alpha}{3-\alpha+2r}\bigg)\Bigg),  $$
	as $\beta \longrightarrow\infty$.
	\hfill$\diamond$
\end{remark}

\subsubsection{Governing partial difference-differential equation}

Under Condition II, let $T>0$ and $Y(t) = \int_0^t g\big(M(s)\big) f(s)\dif s,\ 0\le t\le T,$ where $M\sim CP(\lambda, X)$, $f\in L^1\big([0,T]\big)$ and measurable $g: \mathbb{R}\longrightarrow \mathbb{R}$, $g\in L^1(\mathbb{R})$. Following the line of \cite{GMn1971, Mn1970, P1966, P1968}, we study the equation governing the probability law  $p(t,\dif x, y) \dif y = P\{M(t)\in\dif x, Y(t)\in\dif y\}$ (where the $\dif x$ term must be interpreted as in \eqref{compoundPoissonInfinitesimale}) of the vector process $(M,Y)$. For $t\ge0,\ x,y \in\mathbb{R},$

\begin{align}
	p(t+\dif t,\dif x, y) &= P\{M(t)+M[t,t+\dif t) \in \dif x, Y(t)\in\dif y - \dif t g\big(M(t)\big)f(t) \} \nonumber\\
	& = \int_\mathbb{R} P\{M(t)\in \dif z, Y(t)\in \dif y-g(z)f(t)\}P\{M[t,t+\dif t)\in\dif x-z\} \nonumber\\
	& = \int_\mathbb{R} p\big(t, \dif z, y-\dif t g(z)f(t)\big)\lambda(t)\dif t F(\dif x -z)\nonumber \\
	&\ \ \ +p\big(t, \dif x, y-\dif t g(x) f(t)\big) \big(1-\lambda(t)\dif t\big)+ o(\dif t),\nonumber
\end{align}
and by expanding and keeping the terms of order $\dif t$ or lower,
\begin{align}
	&p(t,\dif x, y) + \frac{\partial}{\partial t}p(t,\dif x,y)\dif t\label{formulaEquazioneDifferenzialePrimoOrdine}\\
	&=  \lambda(t)\dif t\int_\mathbb{R} p(t,\dif x-z, y)F(\dif z) + \big(1-\lambda(t)\dif t)\big)p(t,\dif x,y) -  g(x)f(t) \frac{\partial}{\partial y}p(t,\dif x, y) +o(\dif t).\nonumber
\end{align}
By dividing by $\dif t$ and taking the limit $\dif t\longrightarrow 0$, equation \eqref{formulaEquazioneDifferenzialePrimoOrdine} turns into
	\begin{align}
	\frac{\partial}{\partial t}p(t, \dif x, y)+ g(x)f(t)& \frac{\partial}{\partial y}p(t, \dif x, y) = -\lambda(t) \Big(I-C_F^{(x)}\Big) p(t, \dif x, y) , \ \ \ t\ge0,\ x,y \in\mathbb{R},
\end{align}
where $C_{F}^{(x)} p(t,\dif x, y) =\int_\mathbb{R} p(t,\dif x - z, y) F(\dif z)$.
\\

By suitably adapting the above steps, we derive the following theorem concerning the differential equation governing the $(n+1)$-th dimensional vector describing the compound process and its $n$ iterated integrals.

\begin{theorem}
	Let $T>0$ and $n\in\mathbb{N}$. Let $M\sim CP(\lambda, F)$, measurable real functions $g,f_1,\dots,f_n\in L^1[0,T]$. For $0\le t\le T$, define
	$$Y_1(t) = \int_0^t g\big(M(s)\big) f_1(s) \dif s,\ \ \ Y_k(t) = \int_0^t Y_{k-1}(s)f_k(s)\dif s,\ k=2,\dots, n. $$
	Then, the probability law $p(t,\dif x, y)\dif y = P\{M(t)\in\dif x, Y_1(t)\in \dif y_1,\dots, Y_n(t)\in \dif y_n\}$ satisfies
	\begin{align}
		\frac{\partial}{\partial t}p(t, \dif x, y)+ g(x)f_1(t)& \frac{\partial}{\partial y_1}p(t, \dif x, y) + \sum_{k=2}^n y_{k-1} f_k(t)\frac{\partial}{\partial y_{k}}p(t, \dif x, y) \nonumber\\
		& = - \lambda(t) \Big(I-C_F^{(x)}\Big) p(t, \dif x, y) , \ \ \ t\in [0,T],\ x,y_1,\dots,y_n \in\mathbb{R}, \label{equazioneCompoundPoissonIntegraliIterati}
	\end{align}	
	subject to $p(0,\dif x, y) = \delta_0(\dif x, y)$.
\end{theorem}

Note that the operator on the right-hand side depends on $x$ and $t$ only and it has the same structure of the one in \eqref{equazioneOperatorialeCompoundPoisson}. 

\begin{remark}[Fractional equations and time-changing]
	
In the homogeneous case, by assuming $g(x)=x,x\in\mathbb{R}, f_1\equiv\dots\equiv f_n\equiv1$ and $T=\infty$, it is easy to derive some results concerning the composition with a subordinator or its inverse. In particular, by applying a Bernstein function $f$ to the right-hand side of \eqref{equazioneCompoundPoissonIntegraliIterati}, using the arguments of point ($ii$) of Theorem \ref{teoremaComposizioneCompoundPoissonConSubordinatore} we readily obtain that the solution is still in the class of the compound Poisson processes and it concerns the process $M\circ \mathcal{H}_f$ and its iterated integrals (i.e. of the form $\int_0^t\cdots\dif t_1\int_0^{t_{n-1}} M\circ \mathcal{H}_f(t_n)\dif t_n$), where $\mathcal{H}_f$ is a Bernstein subordinator .
\\

On the other hand, following the arguments of Proposition \ref{proposizioneComposizioneInversoSubordinatore}, the probability law $p_f$ of the time-changed vector $(M\circ L_f, Y_1\circ L_f,\dots, Y_n\circ L_f)$ (i.e. it is not the vector concerning $M\circ L_f$ and its iterated integrals), where $L_f$ is an inverse Bernstein subordinator, satisfies the fractional problem, with $t\ge0,\ x,y_1,\dots,y_n\in \mathbb{R}$,
\begin{equation}\label{equazioneOperatorialeCompoundIntegraliIterati}
	\mathcal{D}_t^f p_f(t,\dif x,y) = -\lambda\Big(I-C_F^{(x)}\Big) p_f(t,\dif x,y) - x \frac{\partial}{\partial y_1}p_f(t, \dif x, y) - \sum_{k=2}^n y_{k-1} \frac{\partial}{\partial y_{k}}p_f(t, \dif x, y),
\end{equation}
subject to $p_f(0,\dif x,y) = \delta_0(\dif x,y)$ and under the hypotheses  of Proposition \ref{proposizioneComposizioneInversoSubordinatore}
and that $\int_0^\infty O_{x,y}p(s, \dif x, y)l_f(t,s)\dif s = O_{x,y} p_f(t,\dif x, y)$, with $O_{x,y}$ being the space operator on the right-hand side of \eqref{equazioneOperatorialeCompoundIntegraliIterati}.
\hfill$\diamond$
\end{remark}

\subsubsection{Homogeneous generalized counting process}

We consider the homogeneous generalized counting process described in Example \ref{esempioSkellam}, i.e. $M = \sum_{i \in\mathcal{I} } i N_i$, with finite $\mathcal{I}\subset (0,\infty)$ and $N_i\sim PP(\lambda_i), \ i\in\mathcal{I}$, being independent homogeneous Poisson processes. 

\begin{proposition}
	Let $M$ be a homogeneous generalized counting process. Then, for $x>0$,
\begin{align}
	P\bigg\{&\int_0^t M(s)\dif s\in \dif x\bigg\} \nonumber\\
	&=  e^{-\lambda t} \Biggl[ \delta_0(\dif x) + \sum_{n\ge1} \frac{1}{n!} \sum_{\substack{k_1,\dots,k_I=0\\k_1+\dots+k_I=n}}^n \binom{n}{k_1,\dots, k_I} \prod_{i\in\mathcal{I}} \Bigg( \bigg(\frac{\lambda_i}{i}\bigg)^{k_{(i)}}  \sum_{h_{(i)} = 0}^{k_{(i)}} \binom{k_{(i)}}{h_{(i)}} (-1)^{h_{(i)}} \Bigg)\nonumber \\
	&\ \ \ \times\mathds{1}_{(t\sum_{j\in\mathcal{I}} j h_{(j)} ,\infty)}(x) \frac{\big(x-t\sum_{j\in\mathcal{I}} j h_{(j)}\big)^{n-1}}{(n-1)!}\dif x\Bigg], \label{probabilitaEsplicitaIntegraleSkemllamGeneralizzato}
\end{align}
where $I = |\mathcal{I}|<\infty$ and, for $i\in\mathcal{I}$, $(i)$ denotes the index of $i$ in the increasingly ordered version of $\mathcal{I}$ (thus, for instance $(\min\{\mathcal{I}\}) = 1$ and $(\max\{\mathcal{I}\}) = I$).
\end{proposition}

Note that the term $e^{-\lambda t}\delta_0$ concerns the case of no Poisson events occurring in $[0,t]$.

\begin{proof}
Write $M$ in the following way as a compound Poisson process, $M\sim CP( \lambda = \sum_{i\in\mathcal{I}} \lambda_i, Y)$, with $P\{Y = i\} = \lambda_i/\lambda$. Now, Remark \ref{remarkIntegraleCompoundOmogeneo} permits us to easily write the moment generating function of the integral of $M$. Let us also assume  $U\sim Unif(0,1)$ independent of $Y$ and $M$. For $t>0,\ \mu>0$, then,
\begin{align}
	\mathbb{E}&\exp\bigg( -\mu \int_0^t M(s)\dif s\bigg) \nonumber\\
	& = \exp\bigg(-\lambda t\big(1-\mathbb{E}e^{-\mu t YU}\big)\bigg)\nonumber \\
	& = \exp\Bigg(-\lambda t \bigg(1- \sum_{i\in\mathcal{I}}\frac{\lambda_i}{\lambda}\, \frac{1-e^{-\mu t i}}{\mu t i}\bigg) \Bigg)\nonumber\\
	%& = e^{-\lambda t}\sum_{n\ge0} \frac{1}{n!} \bigg( \sum_{i\in\mathcal{I}} \lambda_i t \text{sgn}(i) \frac{1-e^{-\mu t i}}{\mu}\bigg)^n\nonumber\\
	& = e^{-\lambda t} \sum_{n\ge0} \frac{\mu^{-n}}{n!}  \Bigg( \sum_{i\in\mathcal{I}} \frac{\lambda_i}{i} \Big(1-e^{-\mu t i}\Big)\Bigg)^n\nonumber\\
	& = e^{-\lambda t} \Bigg[1 + \sum_{n\ge1}\frac{\mu^{-n}}{n!}  \sum_{\substack{k_1,\dots,k_I=0\\k_1+\dots+k_I=n}}^n \binom{n}{k_1,\dots, k_I} \prod_{i\in\mathcal{I}} \bigg(\frac{\lambda_i}{i}\bigg)^{k_{(i)}}  \Big( 1-e^{-\mu t i}\Big)^{k_{(i)}}\Bigg]\nonumber\\
	& = e^{-\lambda t}  \Bigg[1 + \sum_{n\ge1} \frac{\mu^{-n}}{n!}  \sum_{\substack{k_1,\dots,k_I=0\\k_1+\dots+k_I=n}}^n \binom{n}{k_1,\dots, k_I} \prod_{i\in\mathcal{I}} \bigg(\frac{\lambda_i}{i}\bigg)^{k_{(i)}}  \sum_{h_{(i)} = 0}^{k_{(i)}} \binom{k_{(i)}}{h_{(i)}} \Big(-e^{-\mu t i}\Big)^{h_{(i)}} \Bigg]\nonumber\\
	& = e^{-\lambda t}  \Bigg[1 + \sum_{n\ge1} \frac{1}{n!} \sum_{\substack{k_1,\dots,k_I=0\\k_1+\dots+k_I=n}0}^n \binom{n}{k_1,\dots, k_I} \prod_{i\in\mathcal{I}} \bigg(\frac{\lambda_i}{i}\bigg)^{k_{(i)}}  \sum_{h_1 = 0}^{k_1} \binom{k_1}{h_1} (-1)^{h_1}\cdots \sum_{h_I = 0}^{k_I} \binom{k_I}{h_I} (-1)^{h_I} \nonumber\\
	&\ \ \ \times \frac{e^{-\mu t\sum_{j\in\mathcal{I}} jh_{(j)} } }{\mu^n}\Bigg]. \label{trasformataLaplaceIntegraleSkellamGeneralizzato}
\end{align}

Now, we readily obtain the probability density \eqref{probabilitaEsplicitaIntegraleSkemllamGeneralizzato} by inverting \eqref{trasformataLaplaceIntegraleSkellamGeneralizzato} term by term. We use the linearity of the Laplace transform and that, for $\mu, H>0$ and $n\ge 1$,
\begin{equation}\label{trasformazioneLaplaceFunzionePerLeggeEsplicitaSkellam}
	\int_{0}^\infty e^{-\mu x} \mathds{1}_{(H,\infty)}(x)  \frac{(x-H)^{n-1}}{(n-1)!}\dif x = e^{-\mu H}	\int_{0}^\infty e^{-\mu y} \frac{y^{n-1}}{(n-1)!}\dif y = \frac{e^{-\mu H}}{\mu^n},
\end{equation}
which completes the proof.
\end{proof}

In the case of $\mathcal{I} = \{1\}$ some calculations yield the distribution of the integral of the homogeneous Poisson process $N\sim PP(\lambda)$ (see also Section 3 of \cite{SVw2007}), $x>0$,
\begin{align*}
	P\left\{\int_0^t N(s)	\dif s \in \dif x\right\} = e^{-\lambda t}\left[\delta_0(\dif x) + \sum_{n\ge 1} \frac{\lambda^n}{n!}\sum_{h=0}^n\binom{n}{h}(-1)^h (x-ht)^{n-1}\mathds{1}_{(ht, \infty)}(x)\dif x \right].
\end{align*}
 %Note that formula \eqref{leggeIntegralePoissonOmogeneo} can also be derived by means of an induction argument, see Appendix \cite{dimostrazioneLeggeIntegralePoissonOmogeneoRicorrente}.

%we define the compound Poisson process $M(t) = \sum_{k=1}^{N(t)} X_k$ (which is compound Poisson since Proposition \ref{proposizioneCompostoDiCompostoProcessoConteggio}). Here we derive an explicit formula for the probability distribution of the integral of $M$.

\subsection{Renewal case}

We begin by proving some general results involving the Fourier and the Laplace transforms. These are then used to derive results for the integrals of compound renewal processes in the Fourier-Laplace domain.
For suitable $f$ and $\mu>0$, hereafter we use the following notation, $\mathcal{L} \{f\}(\mu) =\mathcal{L} f(\mu)= \int_0^\infty e^{-\mu t} f(t)\dif t$.

\begin{lemma}\label{lemmaTrasformataFourierLaplaceSommeLegateRinnovo}
	Let $N$ be a renewal process with waiting times $W\sim G$ and $\{A_k^{(n)}\}_{k,n\ge1}$ be a sequence of real random variables independent of $N$ (and its waiting times),
	\begin{itemize}
		\item[($i$)] with $\gamma\in\mathbb{R}$,
		\begin{equation}\label{trasformataCompoundRinnovo}
			\mathbb{E} \exp\Biggl( i\gamma \sum_{k=1}^{N(t)} W_k A_k^{(N(t))} \Biggr) = \sum_{n\ge0}\mathbb{E} \Biggl[\conv_{k=1}^n \Big( e^{i\gamma A_k^{(n)}\sbullet} G(\sbullet) \Big)\ast 1 (t)- \conv_{k=1}^{n} \Big( e^{i\gamma A_k^{(n)}\sbullet}G(\sbullet) \Big)\ast G \ast 1 (t) \Biggr],
		\end{equation}
		where $hG\ast f (t)= \int_0^t f(t-x)h(x)G(\dif x)$ for suitable functions $h,f$ and distribution $G$.
		\item[($ii$)] If $W$ is absolutely continuous with density $g$ and $\mu>0$,
		\begin{equation}\label{trasformataLaplaceFourierCompoundRinnovo}
			\mathcal{L} \left\{ \mathbb{E} \exp\Biggl( i\gamma \sum_{k=1}^{N(t)} W_k A_k^{(N(t))}\Biggr) \right\} (\mu) = \frac{1-\mathcal{L}g(\mu)}{\mu}\sum_{n\ge0}\mathbb{E} \prod_{k=1}^n \mathcal{L}g(\mu+i\gamma A_k^{(n)}).
		\end{equation}
	\end{itemize}
\end{lemma}

\begin{proof}
($i$)  We recall that for a renewal process, with $n\ge1$,
\begin{align}\label{leggeCongiuntaTempiProcessoRinnovo}
	P\{W_1\in \dif w_1, \dots, W_n\in \dif w_n, N(t)=n\} = \prod_{k=1}^n P(W_k\in\dif w_k) P(W_{n+1}>t-w_1-\dots-w_n).
\end{align}
Let $\mathcal{A}_n = \text{Supp}(A^{(n)}_1)\times \dots \times\text{Supp}(A^{(n)}_n)$, $\gamma \in \mathbb{R},$ then, 
\begin{align*}
	\mathbb{E} \exp\Biggl( &i\gamma \sum_{k=1}^{N(t)} W_k A_k^{(N(t))} \Biggr) \\
	&= \sum_{n\ge0}\int_{\mathcal{A}_n} P\{A_1^{(n)}\in \dif a_1,\dots, A_n^{(n)}\in\dif a_n\}\int_0^t e^{i\gamma a_1 w_1} G(\dif w_1)\int_0^{t-w_1} e^{i\gamma a_2w_2}G(\dif w_2)\cdots \\
	& \ \ \ \times \cdots \int_0^{t-w_1-\dots - w_{n-1}} e^{i\gamma a_n w_n}P\{X_{n+1}>t-w_1-\dots-w_n\}G(\dif w_n)\\
	& = \sum_{n\ge0} \int_{\mathcal{A}_n} P\{A_1^{(n)}\in \dif a_1,\dots, A_n^{(n)}\in\dif a_n\}\int_0^t e^{i\gamma a_1w_1} G(\dif w_1)\int_0^{t-w_1} e^{i\gamma a_2w_2}G(\dif w_2)\cdots \\
	& \ \ \ \times \cdots \int_0^{t-w_1-\dots - w_{n-1}}e^{i\gamma a_n w_n}G(\dif w_n)\left(1-\int_0^{t-w_1-\dots-w_n}G(\dif w_{n+1})\right)
\end{align*}
which coincides with \eqref{trasformataCompoundRinnovo}.

($ii$) Assuming that $W$ has density $g$,
\begin{align*}
		\mathcal{L} \Bigg\{& \mathbb{E} \exp\Biggl( i\gamma \sum_{k=1}^{N(t)} W_k A_k^{(N(t))} \Biggr) \Bigg\} (\mu) \\
		&= \sum_{n\ge0}\mathbb{E} \Biggl[ \mathcal{L}\left\{ \conv_{k=1}^n \Big( e^{i\gamma A_k^{(n)}\sbullet}G(\sbullet) \Big)\ast 1 \right\}(\mu) - \mathcal{L}\left\{ \conv_{k=1}^n \Big( e^{i\gamma A_k^{(n)}\sbullet} G(\sbullet)\Big)\ast G\ast 1 \right\}(\mu)\Biggr]\\
		& =  \sum_{n\ge0}\mathbb{E} \left[ \frac{1}{\mu} \prod_{k=1}^n \mathcal{L}g(\mu+i\gamma A_k^{(n)}) - \frac{\mathcal{L}g(\mu)}{\mu} \prod_{k=1}^n \mathcal{L}g(\mu+i\gamma A_k^{(n)}) \right],
\end{align*}
which coincides with \eqref{trasformataLaplaceFourierCompoundRinnovo}.
\end{proof}

%Lemma \ref{lemmaTrasformataFourierLaplaceSommeLegateRinnovo} yields the following statement.
Now we obtain the following results for the integral of a compound renewal process.

\begin{theorem}\label{teoremaTrasformateIntegraleCompoundRinnovo}
	Let $M\sim CR(W,X)$,\ $W\sim G$. 
	\begin{itemize}
		\item[($i$)] With $\gamma \in\mathbb{R}$,
		\begin{align}
			\mathbb{E}&\exp\Bigg( i\gamma\bigg(t\sum_{k=1}^{N(t)} X_k - \int_0^t M(s)  \dif s\bigg) \Bigg)\nonumber\\
			&= \sum_{n\ge0} \mathbb{E}  \Biggl[\conv_{k=1}^n \Big(  e^{i\gamma \sum_{h=k}^nX_h\sbullet}G(\sbullet)\Big)\ast 1 (t)- \conv_{k=1}^{n} \Big( e^{i\gamma \sum_{h=k}^nX_h\sbullet}G(\sbullet) \Big)\ast G \ast 1 (t) \Biggr]. \label{fourierIntegraleCompoundRinnovoOsservandoX} 
		\end{align}
		\item[($ii$)] If $W$ is absolutely continuous with density $g$ and $\mu>0$,
		\begin{align}
		\int_0^t e^{-\mu t} \mathbb{E} \exp\Bigg( i\gamma\bigg(t\sum_{k=1}^{N(t)} X_k &- \int_0^t M(s)  \dif s\bigg) \Bigg) \dif t = \frac{1-\mathcal{L}g(\mu)}{\mu}\sum_{n\ge0}\mathbb{E} \prod_{k=1}^n \mathcal{L}g\bigg(\mu+i\gamma \sum_{h=k}^n X_k\bigg). \label{fourierLaplaceIntegraleCompoundRinnovo}
	\end{align}
	\item[($iii$)] With $H_X(\gamma) = \mathbb{E}e^{i\gamma X},\ \gamma \in \mathbb{R}$,
		\begin{equation}
			\mathbb{E} \exp\left( i\gamma \int_0^t M(s)  \dif s\right ) = \sum_{n\ge0} \mathbb{E} \Bigg[ \mathds{1}(W_{n+1}> t-S_n ) \prod_{k=1}^n H_X\big(\gamma (t-S_k)\big)\Bigg], \label{fourierIntegraleCompoundRinnovoOsservandoY}
		\end{equation}
		where $S_k = \sum_{h=1}^k W_h,\ k\in\mathbb{N}$.
	\end{itemize}
\end{theorem}

\begin{proof}
	By formula \eqref{formulazioneGeneraleIntegraleProcessoCompound},
	$$ \int_0^t M(s)  \dif s = \sum_{k=1}^{N(t)} X_k (t-S_k) = t \sum_{k=1}^{N(t)} X_k -\sum_{h=1}^{N(t)} W_h \sum_{k=h}^{N(t)} X_k.$$
	Points ($i$) and ($ii$) of Lemma \ref{lemmaTrasformataFourierLaplaceSommeLegateRinnovo}, with $A_h^{(n)} = \sum_{k=h}^n X_k,\ h,n\ge1$, respectively yield formula \eqref{fourierIntegraleCompoundRinnovoOsservandoX} and formula \eqref{fourierLaplaceIntegraleCompoundRinnovo}.
	Formula \eqref{fourierIntegraleCompoundRinnovoOsservandoY} follows from \eqref{funzioneCaratteristicaIntegraleGenerale} and \eqref{leggeCongiuntaTempiProcessoRinnovo}.
\end{proof}

Theorem \ref{teoremaTrasformateIntegraleCompoundRinnovo} generalizes Theorem 4.1 of \cite{SVw2007}. Indeed, if $X=1$ a.s. formula \eqref{fourierLaplaceIntegraleCompoundRinnovo} concerns the integral of the renewal process $N$ and it coincides with equation (7) of \cite{SVw2007}.

%\section{Maximum likelihood inference}

%%%%%%%%%%%%%%%%%%%%%%%%%%

%\subsection*{\large{Declarations}}
%
%\textbf{Ethical Approval.} This declaration is not applicable.
% \\
%\textbf{Competing interests.}  The authors have no competing interests to declare.
%\\
%\textbf{Authors' contributions.} Both authors equally contributed in the preparation and the writing of the paper.
% \\
%\textbf{Funding.} The authors received no funding.
% \\
%\textbf{Availability of data and materials.} This declaration is not applicable.

%\subsection*{Acknowledgments}

% ---------------------------------------------------------------------------------------------------------


 \appendix

 \section{Proof of formulas \eqref{formulazioneGeneraleIntegraleProcessoTempoContinuo} and \eqref{formulazioneGeneraleIntegraleProcessoCompound}}\label{appendiceCalcoloEsplicitoIntegrali}
 Formula \eqref{formulazioneGeneraleIntegraleProcessoTempoContinuo}.
\begin{align*}
	\int_0^t Y_{N(s)} f(s)\dif s &=  \sum_{n=0}^\infty  Y_n \int_{0}^t \mathds{1}\big(N(s) =n\big) f(s)\dif s \\
	&= \sum_{n=0}^\infty  Y_n \int_{0}^t \mathds{1}\big(S_n \le s < S_{n+1}\big) f(s)\dif s \\
	&= \sum_{n=0}^\infty  Y_n \int_{\min\{S_n, t\}}^{\min\{t, S_{n+1}\}} f(s)\dif s \\
	&= \sum_{n=0}^{N(t)-1}  Y_n \int_{S_n}^{S_{n+1}} f(s)\dif s + Y_{N(t)} \int_{S_{N(t)}}^t f(s)\dif s,
\end{align*}
where in the first step we exchanged series and integral in light of Condition II (which ensures finiteness on bounded time intervals).

Formula \eqref{formulazioneGeneraleIntegraleProcessoCompound}.
\begin{align*}
	\int_0^t  \sum_{k=1}^{N(s)} X_k f(s)\dif s &=  \sum_{n=0}^{N(t)-1}  \sum_{k=1}^{n} X_k \int_{S_n}^{S_{n+1}} f(s)\dif s + \sum_{k=1}^{N(t)} X_k \int_{S_{N(t)}}^t f(s)\dif s \\
	& = \sum_{k=1}^{N(t)-1} X_k  \int_{S_k}^{S_{N(t)}} f(s)\dif s  + \sum_{k=1}^{N(t)} X_k \int_{S_{N(t)}}^t f(s)\dif s\\
	& = \sum_{k=1}^{N(t)} X_k   \int_{S_k}^t f(s)\dif s,
\end{align*}
where in the last equation we used that in the first sum we can add the term $k=N(t)$ since it is equal to $0$.

%\section{Induction proof of Formula \eqref{leggeIntegralePoissonOmogeneoRicorrente}}\label{dimostrazioneLeggeIntegralePoissonOmogeneoRicorrente}

\section{Proof of Lemma \ref{lemmaIncrementiIntegraleProcesso}}\label{appendiceDimostrazioneIncrementiIntegrale}
($i$) For $\gamma\in\mathbb{R}$ and $0\le t_1<t_2\le T$,
\begin{align}
	\mathbb{E}e^{i\gamma \big(I(t_2)-I(t_1)\big)}  &=  \mathbb{E} \exp\Bigg( i\gamma \int_{t_1}^{t_2} \Big(Y(s)\pm Y(t_1) \Big) \dif s \Bigg) \nonumber\\
	& = \mathbb{E} \exp\Bigg( i\gamma \int_{t_1}^{t_2} \Big(Y(s)- Y(t_1) \Big) \dif s  +  i\gamma (t_2 - t_1) Y(t_1)\Bigg) \nonumber
\end{align}
and the independence of the increments yields \eqref{incrementoIntegraleProcessoStocasticoIncrementiIndipendenti}.
\\
Point ($ii$) readily follows from ($i$) by observing that $\{X(s+t_1)-X(t_1)\}_{s\ge0} \stackrel{d}{=}X$ on $\mathcal{D}[0,T]$ because of ($ii$) of Lemma \ref{lemmaUguaglianzaProcessiSpazioFunzioniCadlag}. Hence, the integral of the increment process is equal in distribution to the integral of $X$.

\section{Proof of Proposition \ref{proposizioneVariabiliOrdinateScambiabilita}}\label{appendiceDimostrazioneProposizioneVariabiliOrdinateScambiabilita}

Let $n\ge2$ be an integer. We first prove that, if $h:\mathbb{R}^n\longrightarrow \mathbb{R}$ is commutative in the inputs, i.e. $h(x_1,\dots,x_n) = h(x_{i_1},\dots, x_{i_n})$ for all permutations $i_1,\dots,i_n$ of the indices $1,\dots, n$, then,
\begin{align}
 	\int_0^t \dif x_1 \cdots\int_0^t \dif x_n\, h(x_1,\dots,x_n)& = n! \int_0^t\dif x_1\int_{x_1}^t\dif x_2 \cdots\int_{x_{n-1}}^t \dif x_n \,h(x_1,\dots, x_n)\label{formulaDecomposizioneIntegraleMultiploFunzioneCommutativa}\\
	& =: n! \, I_n(h).\nonumber
\end{align}
For $n=2$ the proof is straightforward. By assuming it holds for $n-1\ge2$, then for $n$ we have (with $x_0=0$),
\begin{align*}
	\int_0^t &\dif x_1 \cdots\int_0^t \dif x_n \,h(x_1,\dots,x_n)\\
	& = (n-1)! \int_0^t\dif x_1\cdots \int_{x_{n-2}}^t \dif x_{n-1} \int_0^t h(x_1,\cdots, x_{n-1}, x_n) \dif x_n\\
	& = (n-1)! \int_0^t\dif x_1\cdots \int_{x_{n-2}}^t \dif x_{n-1} \left(\sum_{k=1}^{n-1} \int_{x_{k-1}}^{x_k} h(x_1,\dots, x_n)\dif x_n + \int_{x_{n-1}}^t h(x_1,\dots, x_n)\dif x_n \right)\\
	& = (n-1)! \,n\, I_n(h), \nonumber
\end{align*}
In the last step each term of the sum has been treated considering that, for $k=1,\dots, n-1$, 
\begin{align*}
	\int_0^t \dif x_1&\cdots\int_{x_{k-2}}^t\dif x_{k-1} \int_{x_{k-1}}^t \dif x_k \int_{x_{k-1}}^{x_k} \dif x_n \int_{x_k}^t\dif x_{k+1}\cdots \int_{x_{n-2}}^t \dif x_{n-1}\, h(x_1,\dots, x_n) \\
	& = \int_0^t \dif x_1\cdots\int_{x_{k-2}}^t\dif x_{k-1} \int_{x_{k-1}}^t\dif x_n \int_{x_n}^t \dif x_k \int_{x_k}^t\dif x_{k+1}\cdots \int_{x_{n-2}}^t \dif x_{n-1}\, h(x_1,\dots, x_n)\\
	& = I_n(h)
\end{align*}
where we used the commutativity $h(x_1,\dots, x_n) = h(x_1,\dots,x_{k-1}, x_n, x_{k},\dots,x_{n-1})$ in the last step.

Now we can prove result \eqref{formulaIndipendentiOrdinatePerCommutazione}. First, let $i:\{1,\dots,n\}\longrightarrow\{1,\dots, n\}$ be a permutation of the indices $1,\dots, n$ and $j$ be its inverse. For any $\lambda_1,\dots,\lambda_n\in [0,t]$ and $\gamma\in\mathbb{R}$, we have 
\begin{align}
	\mathbb{E}\exp\Bigg( & i\gamma g\bigg(f\Big(X_1, \lambda_1\Big),\dots, f\Big(X_n, \lambda_{n}\Big) \bigg) \Bigg)\nonumber \\
	&= \int_{\mathbb{R}^n} \exp\Bigg( i\gamma g\bigg(f\Big(x_1, \lambda_1\Big),\dots, f\Big(x_n, \lambda_{n}\Big) \bigg) \Bigg) P\{X_1\in \dif x_{i_1},\dots,X_n\in \dif x_{i_n}\}\nonumber\\
	&= \mathbb{E}\exp\Bigg( i\gamma g\bigg(f\Big(X_{j_1}, \lambda_1\Big),\dots, f\Big(X_{j_n}, \lambda_{n}\Big) \bigg) \Bigg)\nonumber\\
	& = \mathbb{E}\exp\Bigg( i\gamma g\bigg(f\Big(X_1, \lambda_{i_1}\Big),\dots, f\Big(X_n, \lambda_{i_n}\Big) \bigg) \Bigg) \label{scambiabilitaLambdaInFunzioneCaratteristicaX}
\end{align} 
where in the first step we used the exchangeability of the $Xs$; the second step follows by observing that $X_k\in \dif x_{i_k}$ for all $k=1,\dots, n$, means that $X_{j_k}\in \dif x_k$ for all $k=1,\dots, n$; in the third step we used the commutativity of $g$ (as well as again the fact that $j$ is the inverse of $i$).
Furthermore, we consider the following
\begin{align*}
	h(\lambda_1,\dots,\lambda_n)& := P\{\Lambda_1 \in \dif \lambda_1, \dots, \Lambda_n\in \dif \lambda_n\} \mathbb{E}\exp\Bigg( i\gamma g\bigg(f\Big(X_1, \lambda_1\Big),\dots, f\Big(X_n, \lambda_{n}\Big) \bigg) \Bigg)\nonumber \\
	& = P\{\Lambda_1 \in \dif \lambda_{i_1}, \dots, \Lambda_n\in \dif \lambda_{i_n}\} \mathbb{E}\exp\Bigg( i\gamma g\bigg(f\Big(X_1, \lambda_1\Big),\dots, f\Big(X_n, \lambda_{n}\Big) \bigg) \Bigg)\nonumber\\
	& = P\{\Lambda_1 \in \dif \lambda_{i_1}, \dots, \Lambda_n\in \dif \lambda_{i_n}\} \mathbb{E}\exp\Bigg( i\gamma g\bigg(f\Big(X_1, \lambda_{i_1}\Big),\dots, f\Big(X_n, \lambda_{i_n}\Big) \bigg) \Bigg)\\
	&=h(\lambda_{i_1},\dots,\lambda_{i_n}),
\end{align*}
where in the second step we used the exchangeability of the $\Lambda$s and in the third one we used \eqref{scambiabilitaLambdaInFunzioneCaratteristicaX}.
Thus, the function $h$ is commutative and, using the independence of the $\Lambda$s and the $X$s, we obtain
\begin{align*}
	\mathbb{E}\exp\Bigg( i\gamma &g\bigg(f\Big(X_1, \Lambda_{1}\Big),\dots, f\Big(X_n, \Lambda_{n}\Big) \bigg) \Bigg) \\
	&= \int_0^t\dif \lambda_1 \cdots \int_0^t \dif \lambda_n\,  h(\lambda_1,\dots,\lambda_n)\\
	& = n! \int_0^t \dif \lambda_1 \int_{\lambda_1}^t \dif \lambda_2\cdots \int_{\lambda_{n-1}}^t \dif \lambda_n \,h(\lambda_1,\dots,\lambda_n)\\
	& = \mathbb{E}\exp\Bigg( i\gamma g\bigg(f\Big(X_1, \Lambda_{(1)}\Big),\dots, f\Big(X_n, \Lambda_{(n)}\Big) \bigg) \Bigg),
\end{align*} 
where in the first step we used the independence assumption, in the second one we used \eqref{formulaDecomposizioneIntegraleMultiploFunzioneCommutativa} and in the last one that $P\{\Lambda_{(1)}\in\dif\lambda_1,\dots,\Lambda_{(n)}\in \dif \lambda_n\} = n!P\{\Lambda_{1}\in\dif\lambda_1,\dots,\Lambda_{n}\in \dif \lambda_n\}$ with $0\le \lambda_1<\dots<\lambda_n\le t$.

% \section{}
% \label{}

%%       \cite{<label>} ==> Jones et al. [21]
%%       \citep{<label>} ==> [21]
%%


\footnotesize{

}

\end{document}